\documentclass[11pt,reqno]{amsart}
\usepackage{amsmath}
\usepackage{amsthm}
\usepackage{graphicx}
\usepackage{amsfonts}
\usepackage{amssymb}
\usepackage{hyperref}
\usepackage{bm}
\usepackage{color}
\usepackage[margin=1in]{geometry}
\usepackage{enumerate}
\usepackage{mathrsfs}

\numberwithin{equation}{section}

\newtheorem{theorem}{Theorem}[section]
\newtheorem{lemma}[theorem]{Lemma}
\newtheorem{proposition}[theorem]{Proposition}
\newtheorem{corollary}[theorem]{Corollary}

\theoremstyle{definition}
\newtheorem{example}[theorem]{Example}
\newtheorem{remark}[theorem]{Remark}

\def\E{{\mathbb E}}
\def\R{{\mathbb R}}
\def\N{{\mathbb N}}
\def\PP{{\mathbb P}}
\def\FF{{\mathbb F}}
\def\P{{\mathcal P}}
\def\T{{\mathbb T}}
\def\W{{\mathcal W}}
\def\H{{\mathcal H}}
\def\F{{\mathcal F}}
\def\tr{{\mathrm{Tr}}}
\def\Var{{\mathrm{Var}}}

\title[Independent projections of diffusions]{Independent projections of diffusions: Gradient flows for variational inference and  optimal mean field approximations}
\author{Daniel Lacker} 
\address{Department of Industrial Engineering \& Operations Research, Columbia University}
\email{daniel.lacker@columbia.edu}
\thanks{D.L.\ is partially supported by the NSF CAREER award DMS-2045328.}

\begin{document}
\begin{abstract}
What is the optimal way to approximate a high-dimensional diffusion process by one in which the coordinates are \emph{independent}? This paper presents a construction, called the \emph{independent projection}, which is optimal for two natural criteria. First, when the original diffusion is reversible with invariant measure $\rho_*$, the independent projection serves as the Wasserstein gradient flow for the relative entropy $H(\cdot\,|\,\rho_*)$ constrained to the space of product measures. This is related to recent Langevin-based sampling schemes proposed in the statistical literature on mean field variational inference. In addition, we provide both qualitative and quantitative results on the long-time convergence of the independent projection, with quantitative results in the log-concave case derived via a new variant of the logarithmic Sobolev inequality. Second, among all processes with independent coordinates, the independent projection is shown to exhibit the slowest growth rate of path-space entropy relative to the original diffusion. This sheds new light on the classical McKean-Vlasov equation and recent variants proposed for non-exchangeable systems, which can be viewed as special cases of the independent projection.
\end{abstract}

\maketitle

\section{Introduction}

Fix positive integers $d$ and $n$, thinking of $n$ as the number of particles and $d$ as the dimension of the space in which each particle resides.
Fix a Lipschitz function  $\bm{b}=(b^1,\ldots,b^n) : (\R^d)^n \to (\R^d)^n$.
Suppose particles $(Y^1_t,\ldots,Y^n_t)=\bm{Y}_t$ satisfy the stochastic different equation (SDE)
\begin{equation}
d\bm{Y}_t = \bm{b}(\bm{Y}_t)\,dt + \sqrt{2}\,d\bm{B}_t, \label{def:SDE-reference}
\end{equation}
where $\bm{B}=(B^1,\ldots,B^n)$ is a standard Brownian motion in $(\R^d)^n \cong \R^{dn}$. The driving Brownian motions are independent, but the vector field $b$ generally induces dependence between the particles. 

In this paper, we study a process $\bm{X}=(X^1,\ldots,X^n)$ with \emph{independent} components which, in certain precise senses, is a good approximation of  $\bm{Y}$. 
We first need some notation.
For $\bm{x}=(x^1,\ldots,x^n) \in (\R^d)^n \cong \R^{dn}$, $y \in \R^d$, and  $i=1,\ldots,n$, we  denote ``$\bm{x}$ with the $i$th coordinate removed" and ``$\bm{x}$ with the $i$th coordinate swapped for $y$," respectively, by
\begin{align*}
\bm{x}^{-i} &:= (x^1,\ldots,x^{i-1},x^{i+1},\ldots,x^n) \in \R^{d(n-1)}, \\
(y,\bm{x}^{-i}) &:= (x^1,\ldots,x^{i-1},y,x^{i+1},\ldots,x^n) \in \R^{dn}.
\end{align*}
We use similar notation for measures: for a probability measure $\mu$ on $\R^{dn}$, we write $\mu^{-i}$ for the  pushforward of $\mu$ through the map $\bm{x} \mapsto \bm{x}^{-i}$.
The SDE we study is the following:
\begin{equation}
\begin{cases}
dX^i_t &= \int_{\R^{d(n-1) }} b^i(X^i_t,\bm{x}^{-i})\,\mu^{-i}_t(d\bm{x}^{-i})\,dt + \sqrt{2}\,dB^i_t, \\
\mu_t &= \mathrm{Law}(\bm{X}_t), \quad t \ge 0, \ \ i=1,\ldots,n.
\end{cases} \label{def:mainSDE}
\end{equation}
We term this SDE \emph{the independent projection} of \eqref{def:mainSDE}, for reasons that will be explained below. Because we assume $\bm{b}$ to be Lipschitz, the existence and uniqueness of a solution of  \eqref{def:mainSDE} is a straightforward application of known results on McKean-Vlasov equations; see Proposition \ref{pr:wellposed}.

This is a \emph{McKean-Vlasov} equation, in the sense that the drift depends not only on $\bm{X}_t$ but also on the law $\mu_t$ of the solution. The only stochastic term entering the drift of $dX^i_t$ is $X^i_t$ itself, and not $(X^j_t)_{j \ne i}$ which are integrated out. Because of this, if the initial positions are assumed independent, then the processes $X^1,\ldots,X^n$ are necessarily independent.
This explains the term \emph{independent} in our title \emph{independent projection} and also allows us to rewrite \eqref{def:mainSDE} in a more probabilistic notation,
\begin{equation}
dX^i_t = \E[b^i(\bm{X}_t)\,|\,X^i_t]\,dt + \sqrt{2}\,dB^i_t, \quad i=1,\ldots,n. \label{def:mainSDE-probabilistic}
\end{equation}
The conditional expectation provides a first justification of our choice of term \emph{projection} in \emph{independent projection}. In Section \ref{se:geometry} we will see a deeper (though formal) justification in terms of the geometry of Wasserstein space: the drift of $d\bm{X}_t$ can be viewed the $L^2(\mu_t)$-projection of the vector field $\bm{b}$ onto the tangent space at $\mu_t$ of the submanifold of \emph{product measures}.

The time-marginals of the processes $\bm{Y}$ and $\bm{X}$ may be described in terms of partial differential equation (PDEs).
As is well known, the law $\rho_t$ of $\bm{Y}_t$ evolves according to the Fokker-Planck equation
\begin{equation}
\partial_t \rho_t = -\mathrm{div}(\rho_t \bm{b}) + \Delta \rho_t. \label{def:FokkerPlanck}
\end{equation}
The law $\mu^i_t$ of $X^i_t$, for $i=1,\ldots,n$, similarly evolves according to a system of coupled PDEs,
\begin{equation}
\partial_t \mu^i_t = -\mathrm{div}(\mu^i_t \widehat{b}^i[\mu_t]) + \Delta \mu^i_t. \label{def:newFokkerPlanck-system}
\end{equation}
Here $\widehat{b}^i[\mu_t] :  \R^d \to \R^d$ derives from $b^i$ by integrating out the variables $j \neq i$: 
\begin{equation*}
\widehat{b}^i[\mu_t](x^i) := \int_{\R^{d(n-1)}} b^i(x^i,\bm{x}^{-i})\,\mu^{-i}_t(d\bm{x}^{-i}), \quad t \ge 0, \ x^i \in \R^d.
\end{equation*}
Letting $\widehat{\bm{b}}[\mu_t](\bm{x})=(\widehat{b}^1[\mu_t](x^1),\ldots,\widehat{b}^n[\mu_t](x^n))$, it is easy to deduce that the law $\mu_t=\mu^1_t\otimes\cdots\otimes\mu^n_t$ of $\bm{X}_t$   follows a nonlinear, nonlocal PDE,
\begin{equation}
\partial_t \mu_t = -\mathrm{div}(\mu_t \widehat{\bm{b}}[\mu_t]) + \Delta \mu_t. \label{def:newFokkerPlanck}
\end{equation}

\begin{example} \label{ex:McKeanVlasov}
A widely studied class of examples arises when the drift is a self-interaction plus a (mean field) pairwise interaction,
\begin{equation*}
b^i(\bm{x}) = K_1(x^i) + \frac{1}{n-1}\sum_{j\neq i}K_2(x^i,x^j).
\end{equation*}
The initial SDE \eqref{def:SDE-reference} then becomes a classical interacting particle system,
\begin{equation}
dY^i_t = \bigg(K_1(Y^i_t) + \frac{1}{n-1}\sum_{j \neq i}K_2(Y^i_t,Y^j_t)\bigg)dt + \sqrt{2}\,dB^i_t, \quad i=1,\ldots,n. \label{def:meanfieldparticles}
\end{equation}
If the initial positions $Y^1_0,\ldots,Y^n_0$ are iid $\sim \mu_0$, then the phenomenon of \emph{propagation of chaos} \cite{sznitman1991topics} dictates that  the empirical measure $\frac{1}{n}\sum_{i=1}^n\delta_{Y^i_t}$ converges as $n\to\infty$ to a non-random limit $\overline{\mu}_t$, which is characterized as the solution of the McKean-Vlasov equation
\begin{equation}
d\overline{Y}_t = \bigg(K_1(\overline{Y}_t) + \int_{\R^d} K_2(\overline{Y}_t,x)\,\overline{\mu}_t(dx)\bigg)dt + \sqrt{2}\,d\overline{B}_t, \quad \overline{Y}_0 \sim \mu_0, \ \overline{\mu}_t=\mathrm{Law}(\overline{Y}_t). \label{def:McKeanVlasov}
\end{equation}
The independent projection turns out to coincide with this McKean-Vlasov equation, as will be shown in Proposition \ref{pr:McKeanVlasov}.
Precisely, when initialized from $\bm{X}_0 \sim \mu_0^{\otimes n}$, the unique solution $\bm{X}$ of \eqref{def:mainSDE} has the same law as $n$ iid copies of $\overline{Y}$, and in particular $\mu_t=\overline{\mu}_t^{\otimes n}$ is the time-$t$ law.
In this sense, the notion of independent projection provides a \emph{non-asymptotic} connection between the particle system \eqref{def:meanfieldparticles} and the McKean-Vlasov equation.
\end{example}

There are three main goals of this paper:

\subsection{Independent projection as a gradient flow} 

Suppose that $\bm{b}=\nabla f$ is the gradient of a function $f$ such that $\rho_*(dx)=e^{f(x)}dx$ is a probability measure.
It is known from the work of Jordan-Kinderlehrer-Otto \cite{jordan1998variational} that $(\rho_t)_{t \ge 0}$ can be interpreted as the curve of steepest descent, or \emph{gradient flow}, in Wasserstein space for the relative entropy functional $H(\cdot\,|\,\rho_*)$.
We show that $(\mu_t)_{t \ge 0}$ is the Wasserstein gradient flow for the same relative entropy functional \emph{constrained to the subset of product measures}.
See Section \ref{se:intro:GF}.
To shed further light on this, Section \ref{se:geometry} describes the (formal) geometric structure of the space of product measures as a submanifold of the Wasserstein space, in the spirit of Otto calculus.  This will explain also a ``complementary" relationship between the independent projection and the \emph{projected Langevin dynamics} recently introduced by the author with G.\ Conforti and S.\ Pal in \cite{ConfortiLackerPal} in connection with entropic optimal transport.

\subsection{Long-time behavior}
It is well known that $\rho_t \to \rho_*$ weakly as $t\to\infty$, with $\rho_*$ of course being the unique minimizer of the entropy $H(\cdot\,|\,\rho_*)$ over the space of probability measures.
We show that the long-time limit points of $\mu_t$ are critical points for the problem of minimizing $H(\cdot\,|\,\rho_*)$ \emph{over the space of product measures}. When $f$ is concave, these critical points are in fact global minimizers. When $f$ is strictly concave there is a unique minimizer, and $\mu_t$ converges to it. When $f$ is strongly concave, convergence holds at an exponential rate, both in Wasserstein distance and in relative entropy, the latter encoded by a new variant of the logarithmic Sobolev inequality. See Section \ref{se:longtime}. 
The statistical literature on \emph{variational inference} contains a great deal of work on the  problem of algorithmically minimizing $H(\cdot\,|\,\rho_*)$ over product measures. 
The continuous-time process studied in this paper would be a natural basis for a sampling approach to mean field variational inference. We discuss some related literature and speculation in Section \ref{se:VIliterature}, though we do not attempt a rigorous analysis of any time-discretizations in this paper.

\subsection{Infinitesimal optimality of the independent projection} \label{se:intro:entropicoptimality}

Consider the relative entropy $H_t$ between the law of $(\bm{X}_s)_{s \le t}$ and the law of $(\bm{Y}_s)_{s \le t}$, viewed as probability measures on the space of continuous paths $[0,t] \to \R^{dn}$.
Consider the infinitesimal growth rate $\liminf_{t\to 0} (H_t-H_0)/t$.
Compared to any other process with independent components and the same time-zero distribution, we  show that $\bm{X}$ minimizes the growth rate $H_0'$.
Recalling Example \ref{ex:McKeanVlasov}, this  translates immediately to a new non-asymptotic optimality property for the McKean-Vlasov equation \eqref{def:McKeanVlasov}, which is a special case of an independent projection.
This property also explains why independent projections have appeared useful in some very recent studies on mean field approximations and propagation of chaos for non-exchangeable interacting particle systems.
See Section \ref{se:intro:entropic} for details.

This explains a sense in which $\bm{X}$ is the \emph{best} approximation of $\bm{Y}$ by independent processes. But this does not mean it is always a \emph{good} approximation. In Section \ref{se:intro:goodbound} we record a quantitative estimate on the distance between the laws of $\bm{X}$ and $\bm{Y}$, along with an application to quantitative mean field limits for non-exchangeable interacting particle systems.

\section{Main results}

This section states the main results of the paper in detail, with most proofs being deferred to later sections.
We begin by summarizing some frequently used notation.
We write $\P(E)$ for the set of Borel probablity measures on a complete separable metric space $(E,d)$, equipped with the topology of weak convergence. We write also $\P^{\otimes n}(E)$ for the subspace of $\P(E^n)$ consisting of $n$-fold product measures, i.e.,
\begin{equation*}
\P^{\otimes n}(E) = \{\mu^1 \otimes \cdots \otimes \mu^n : \mu^i \in \P(E), \ i =1,\ldots,n\}.
\end{equation*}
Note that $\P^{\otimes n}(E)$ and $(\P(E))^n$ are homeomorphic in the obvious way, but they are not isomorphic as subsets of vector spaces; the latter is convex, whereas the former is not. 
The relative entropy between $\mu,\nu \in \P(E)$ is defined as usual by
\begin{align*}
H(\nu\,|\,\mu) = \int\frac{d\nu}{d\mu}\log\frac{d\nu}{d\mu}\,d\mu \ \text{ if } \nu \ll \mu, \quad H(\nu\,|\,\mu) =  \infty \text{ otherwise}.
\end{align*}
For $p \ge 1$ we define $\P_p(E)$ to be the set of $\mu \in \P(E)$ satisfying $\int_Ed^p(x,x_0)\,m(dx) < \infty$ for some $x_0 \in E$. The $p$-Wasserstein distance is defined for  $\mu,\nu \in \P_p(E)$ by
\begin{equation*}
\W_p^p(\mu,\nu) = \inf_{\pi \in \Pi(\mu,\nu)} \int_{E \times E} d^p(x,y)\pi(dx,dy),
\end{equation*}
where $\Pi(\mu,\nu')$ is the set of couplings, i.e., probability measures on $E \times E$ with first marginal $\mu$ and second marginal $\nu$.
Finally, we define
\begin{equation*}
\P_p^{\otimes n}(E) := \P^{\otimes n}(E) \cap \P_p(E^n).
\end{equation*}
A probability measure on a Euclidean space will be tacitly identified with its (Lebesgue) density, when it exists.
Letting $\R_+:=[0,\infty)$, we equip the continuous path space $C(\R_+;E)$ with the usual topology of uniform convergence on compacts.

\begin{remark}
Throughout the paper, the drift $\bm{b}$ is assumed to be Lipschitz mainly for simplicity, and the results can likely be extended to more general settings.
\end{remark}

\subsection{Independent projection as a gradient flow} \label{se:intro:GF}

In this section we assume $\bm{b}=\nabla f$, where $f : \R^{dn} \to \R$ is $C^1$, has Lipschitz gradient, and satisfies $Z := \int e^f < \infty$, so that
\begin{equation}
\rho_*(d\bm{x}) = (1/Z)e^{f(\bm{x})}d\bm{x} \label{def:P}
\end{equation}
defines a probability measure. 
The measure $\rho_*$ is well known to be the unique invariant measure of the Markov process defined by the SDE \eqref{def:SDE-reference}, which rewrites as
\begin{equation}
d\bm{Y}_t = \nabla f(\bm{Y}_t)\,dt + \sqrt{2}d\bm{B}_t. \label{def:SDE-Langevin}
\end{equation}
There are many senses, qualitative and qualitative, in which the solution of \eqref{def:SDE-Langevin} is known to converge to $\rho_*$ as $t\to \infty$. We do not attempt to survey the vast literature here, but we mention the excellent monograph-in-progress \cite{ChewiBook} which documents many known results, techniques, and applications to sampling.

Let us recall the pioneering result of Jordan-Kinderlehrer-Otto (JKO) \cite{jordan1998variational}, which showed in a rigorous sense that the measure flow induced by the Langevin dynamics \eqref{def:SDE-Langevin} is the steepest descent (or gradient flow) on the Wasserstein space $(\P_2(\R^{dn}),\W_2)$ for the entropy functional $H(\cdot\,|\,\rho_*)$. Fix an initial condition $\rho_0 \in \P_2(\R^{dn})$. Let $\tau > 0$. Set $\rho^{\tau,0}=\rho_0$, and for integers $k \ge 0$ let
\begin{equation}
\rho^{\tau,k+1} \in \mathrm{argmin}_{\rho \in \P_2(\R^{dn})} \bigg( H(\rho\,|\,\rho_*) + \frac{1}{2\tau}\W_2^2(\rho,\rho^{\tau,k})\bigg). \label{def:JKO-original}
\end{equation}
A unique minimizer exists, because $H(\cdot\,|\,\rho_*)$ has weakly compact sublevel sets and is strictly convex, and because $\W_2^2(\cdot,\rho^{\tau,k})$ is lower semicontinuous and convex by Kantorovich duality; see also \cite[Proposition 4.1]{jordan1998variational} for a proof from first principles. Define piecewise-constant interpolations by setting $\rho^\tau_0=\rho_0$ and $\rho^\tau_t=\rho^{\tau,k+1}$ for $t \in (k\tau,(k+1)\tau]$. Then, as $\tau \downarrow 0$, $(\rho^\tau_t)_{t \ge 0}$ converges pointwise (weakly) to $(\rho_t)_{t \ge 0}$, where $\rho_t$ is the law of $\bm{Y}_t$ which solves the SDE \eqref{def:SDE-Langevin} initialized from $\bm{Y}_0 \sim \rho_0$.

We will show that the independent projection is similarly the steepest descent when the optimization in \eqref{def:JKO-original} is constrained to the set of product measures.   
Fix an initial condition $\mu_0 \in \P_2^{\otimes n}(\R^d)$ admitting a density. Let $\tau > 0$. Set $\mu^{\tau,0}=\mu_0$, and for integers $k \ge 0$ let
\begin{equation}
\mu^{\tau,k+1} \in \mathrm{argmin}_{\mu \in \P_2^{\otimes n}(\R^d)} \bigg( H(\mu\,|\,\rho_*) + \frac{1}{2\tau}\W_2^2(\mu,\mu^{\tau,k})\bigg). \label{def:JKO}
\end{equation}
A minimizer exists but may not be unique, because $\P_2^{\otimes n}(\R^d) \subset \P_2(\R^{dn})$ is not convex. 
Define piecewise-constant interpolations $(\mu^\tau(t))_{t \ge 0}$ by setting  $\mu^\tau_0=\mu_0$ and $\mu^\tau_t = \mu^{\tau,k+1}$ for $t\in (k\tau,(k+1)\tau]$. The following is proven in Section \ref{se:proof:JKO}.

\begin{theorem} \label{th:JKO}
Let $\bm{X}$ denote the unique solution of \eqref{def:mainSDE} with $\bm{b}=\nabla f$ initalized from $\bm{X}_0 \sim \mu_0$, and let $\mu_t$ be the law of $\bm{X}_t$. 
Then, as $\tau \downarrow 0$, $\mu^\tau_t \to \mu_t$ weakly for each $t \ge 0$.
\end{theorem}

In other words, Theorem \ref{th:JKO} gives a rigorous sense in which the independent projection provides the curve of steepest descent for the following minimization problem:
\begin{equation}
\inf\{H(\mu\,|\,\rho_*) : \mu \in \P^{\otimes n}(\R^d)\}. \label{def:meanfield}
\end{equation}
This minimization problem is well known in probability as a  \emph{mean field approximation}.
In the  statistics literature, it goes by the name  \emph{mean field variational inference} \cite{blei2017variational}; the term \emph{variational inference} refers to the minimization of relative entropy over some class of measures, and the term \emph{mean field} refers to the specification of this class to be the set of product measures.
The general idea is that a high-dimensional measure may not itself be easy to work with, and variational inference identifies an approximation with good properties.
For an alternative perspective, letting $H(\mu)=\int\mu\log\mu$ denote the differential entropy of  $\mu$, the infimum in \eqref{def:meanfield} is precisely
\begin{align*}
\sup_{\mu \in \P(\R^{dn})}\bigg(\int f\,d\mu - H(\mu)\bigg)  - \sup_{\mu \in \P^{\otimes n}(\R^d)}\bigg(\int f\,d\mu - H(\mu)\bigg),
\end{align*}
which quantifies the degree to which the Gibbs variational formula is saturated by product measures. Sharp bounds on this quantity are central in the recent literature on \emph{nonlinear large deviations} initiated by \cite{ChatterjeeDembo}.

The study of constrained gradient flows in Wasserstein goes back to the work of Carlen and Gangbo \cite{CarlenGangbo,carlen2004solution}, which studied a geodesically non-convex subset of Wasserstein space in part to construct solutions of kinetic Fokker-Planck equations which conserve energy and momentum.

\subsection{Long-time behavior} \label{se:longtime}

Our next set of results describe the long-time limit points of the independent projection in terms of the optimization problem \eqref{def:meanfield}, again in the case where $\bm{b}=\nabla f$ is a gradient.
We begin with the cleanest case of concave $f$.
Throughout, we assume $f$ has bounded continuous second derivatives, and we let $\mu_t$ for $t \ge 0$ denote the law of the unique solution $\bm{X}_t$ of \eqref{def:meanfield}, subject to some given initial law $\mu_0 \in \P_2^{\otimes n}(\R^d)$.
Recall that $f$ is said to be $\kappa$-concave if the eigenvalues of $\nabla^2f(\bm{x})$ are at most $-\kappa $ for each $\bm{x}$.

\begin{theorem} \label{th:convex-limit}
Suppose there exist $c_1 \in \R$ and $c_2 > 0$ such that
\begin{equation}
\bm{x} \cdot \nabla f(\bm{x}) \le c_1 - c_2|\bm{x}|^2, \quad \forall \bm{x} \in \R^{dn}. \label{asmp:dissipative}
\end{equation}
\begin{enumerate} 
\item If $f$ is concave, then for any divergent sequence of positive reals $(t_n)$, there exists a subsequence $(t_{n_k})$ such that $\mu_{t_{n_k}}$ converges in $(\P_2(\R^{dn}),\W_2)$ to an optimizer of \eqref{def:meanfield}.
\item If $f$ is strictly concave, then there exists a unique optimizer $\mu_*$ for \eqref{def:meanfield}, and $\W_2(\mu_t,\mu_*) \to 0$.
\item If $f$ is $\kappa$-concave for some $\kappa > 0$, then with $\mu_*$ as in (ii), we have
\begin{equation*}
\W_2(\mu_t,\mu_*) \le e^{-\kappa t}\W_2(\mu_0,\mu_*), \qquad \forall t \ge 0.
\end{equation*}
\end{enumerate}
\end{theorem}

The proof is given in Section \ref{se:proof:longtime}. Part (iii) is a remarkably direct adaptation of the classical proof for Langevin dynamics.
Parts (i) and (ii) are less direct and rely on our next result, which treats the non-concave case.
The difficulty is that with just concavity but not strict concavity, there can be multiple optimizers in \eqref{def:meanfield}. And in the absence of any concavity, there may be local optimizers or critical points which are not global optimizers.
This well known nonuniqueness accounts for the appearance of phase transitions in many mean field models, such as the classical Curie-Weiss model.
Note, of course, that uniqueness of optimizers does not follow from strict convexity of $H(\cdot\,|\,\rho_*)$, because $\P^{\otimes n}(\R^d)$ is not a convex subset of $\P(\R^{dn})$. We will see below, however, that $\P_2^{\otimes n}(\R^d)$ is \emph{geodesically convex} in $\P_2(\R^{dn})$, which is why \eqref{def:meanfield} is well-behaved for  concave $f$.

Regardless of concavity, there is a well known first-order condition for optimality in \eqref{def:meanfield}: we say that a product measure $\mu=\mu^1\otimes \cdots \otimes \mu^n$ solves the \emph{mean field equations} if $f \in L^1(\mu)$ and
\begin{equation}
\mu^i(dx^i) = (1/Z_i)\exp\big(\E_\mu[f(\bm{X})\,|\,X^i=x^i]\big)dx^i,\quad i=1,\ldots,n, \label{def:MFequations}
\end{equation}
where $\bm{X} \sim \mu$ in the expectation.
See, e.g., \cite[Equation (17)]{blei2017variational} or \cite[Equation (1.3)]{LackerMukherjeeYeung}.
Any optimizer of \eqref{def:meanfield} can be shown to satisfy \eqref{def:MFequations} \cite[Proposition 3.5]{LackerMukherjeeYeung}, but the converse is not true in general.
It was recently shown that if $f$ is concave then optimizers of \eqref{def:meanfield} are equivalent to solutions of the mean field equations \eqref{def:MFequations}, and if $f$ is strictly concave then there is a unique optimizer; see Proposition 3.9 and Lemma 3.6 of  \cite{LackerMukherjeeYeung}. Because of this, Theorem \ref{th:convex-limit}(i,ii) follow immediately from the following general result, which is proven in Section \ref{se:proof:longtime}:

\begin{theorem} \label{th:main-MF}
Suppose the condition \eqref{asmp:dissipative} holds.
For any divergent sequence of positive reals $(t_n)$, there exists a subsequence $(t_{n_k})$ such that $\mu_{t_{n_k}}$ converges weakly to a solution $\mu_*$ of the mean field equations \eqref{def:MFequations}.
Moreover, along the same subsequence, it holds that
\begin{equation}
\lim_{k\to\infty} H(\mu_{t_{n_k}}\,|\,\rho_*) = H(\mu_*\,|\,\rho_*). \label{eq:entropyconvergence}
\end{equation}
\end{theorem}

The idea of the proof is to analyze the time-decay of $H(\mu_t\,|\,\rho_*)$. We will see formally that
\begin{equation}
\frac{d}{dt}H(\mu_t\,|\,\rho_*) = -\sum_{i=1}^n\E\big|\E[\nabla_i f(\bm{X}_t)\,|\,X^i_t] - \nabla_i\log \mu^i_t(X^i_t)\big|^2. \label{eq:intro:entropydecay}
\end{equation}
We thus expect any limit point $\mu$ of $\mu_t$ as $t \to \infty$ to make the right-hand side vanish, or
\begin{equation*}
\nabla_i \log\mu^i(x^i) = \E_\mu[\nabla_i f(\bm{X})\,|\,X^i=x^i],
\end{equation*}
where $\bm{X} \sim \mu$ in the expectation.
Antidifferentiating yields the mean field equations \eqref{def:MFequations}.

The entropy decay formula \eqref{eq:intro:entropydecay} naturally suggests a functional inequality to quantify the rate of convergence. Noting that $\nabla_i\log\mu_t(\bm{X}_t)=\nabla_i\log\mu^i_t(X^i_t)$ because $\mu_t(\bm{x})=\mu^1_t(x^1)\cdots\mu^n_t(x^n)$ is a product measure, the right-hand side of \eqref{eq:intro:entropydecay} can be rewritten as
\begin{equation*}
-\sum_{i=1}^n\E\bigg[\bigg|\E\Big[\nabla_i\log \frac{d\mu_t}{d\rho_*}(\bm{X}_t) \, \Big|\, X^i_t\Big]\bigg|^2 \bigg].
\end{equation*}
This motivates the following definition of a \emph{projected Fisher information} functional,
\begin{equation}
\widetilde{I}(\mu\,|\,\rho_*) := \sum_{i=1}^n\E_\mu\bigg[\bigg|\E_\mu\Big[\nabla_i\log \frac{d\mu}{d\rho_*}(\bm{X}) \, \Big|\, X^i\Big]\bigg|^2 \bigg], \qquad \mu=\mu^1\otimes \cdots \otimes \mu^n \in \P^{\otimes n}(\R^d), \label{def:projectedFisher}
\end{equation}
where we set $\widetilde{I}(\mu\,|\,\rho_*)=\infty$ whenever the weak gradient $\nabla \log d\mu/d\rho_*$ fails to exist or to belong to $L^2(\mu)$. We will say that $\rho_*$ satisfies the \emph{projected log-Sobolev inequality} (projected LSI) with constant $\kappa$ if
\begin{equation}
\widetilde{H}(\mu\,|\,\rho_*) := H(\mu\,|\,\rho_*) - \inf_{\nu \in \P^{\otimes n}(\R^d)}H(\nu\,|\,\rho_*) \le \frac{1}{2\kappa}\widetilde{I}(\mu\,|\,\rho_*), \qquad \forall \mu \in \P^{\otimes n}(\R^d). \label{def:projectedLSI}
\end{equation}
If this inequality holds, then from \eqref{eq:intro:entropydecay} we deduce that 
\begin{equation}
\widetilde{H}(\mu_t\,|\,\rho_*) \le e^{-2\kappa t}\widetilde{H}(\mu_0\,|\,\rho_*), \qquad \forall t \ge 0, \label{def:entropyrate}
\end{equation}
which quantifies the rate of convergence.

\begin{theorem} \label{th:LSI}
Suppose $f$ is $\kappa$-concave for some $\kappa > 0$. Then the projected LSI \eqref{def:projectedLSI} holds, as well as the exponential convergence \eqref{def:entropyrate}.
\end{theorem}

The proof, given in Section \ref{se:LSI:proofs}, follows the Bakry-\'Emery strategy \cite{BakryEmery} of computing the time-derivative of $\widetilde{I}(\mu_t\,|\,\rho_*)$ and bounding it in terms of $\widetilde{I}(\mu_t\,|\,\rho_*)$ itself.
Recall for the sake of comparison that the usual Fisher information is defined for any $\mu \in \P(\R^{dn})$, not just product measures, by
\begin{equation}
I(\mu\,|\,\rho_*) := \E_\mu\bigg|\nabla \log \frac{d\mu}{d\rho_*}(\bm{X})\bigg|^2 = \sum_{i=1}^n\E_\mu\bigg[\bigg| \nabla_i\log \frac{d\mu}{d\rho_*}(\bm{X}) \bigg|^2 \bigg], \label{def:FisherInfo}
\end{equation}
similarly set to $\infty$ when the gradient is not well defined.
The usual log-Sobolev inequality (LSI) takes the form
\begin{equation}
H(\mu\,|\,\rho_*)  \le \frac{1}{2\kappa}I(\mu\,|\,\rho_*), \qquad \forall \mu \in \P(\R^{dn}). \label{def:LSI}
\end{equation}
It similarly characterizes the exponential convergence in entropy of the solution of \eqref{def:SDE-Langevin}.
Neither the projected LSI nor the LSI can be deduced from the other.

The results of this section closely parallel the results of the recent paper \cite{ConfortiLackerPal} by the author with G.\ Conforti and S.\ Pal. The paper \cite{ConfortiLackerPal} studies an SDE, termed the \emph{projected Langevin dynamics}, which relates to the space of couplings $\Pi(\mu,\nu)$ of given marginals $(\mu,\nu)$ in the same way that the independent projection relates to the space of product measures $\P^{\otimes n}(\R^d)$. Specifically, the law of their dynamics at time $t$ converges as $t \to \infty$ to the solution of an entropy-minimization problem constrained to $\Pi(\mu,\nu)$, known as \emph{entropic optimal transport}. Our proof of Theorem \ref{th:main-MF} follows a similar strategy to the one developed in \cite{ConfortiLackerPal}. A similar projected LSI is derived in \cite{ConfortiLackerPal} in certain regimes, but by completely different arguments because the Bakry-\'Emery strategy did not appear to be tractable there. The difference between these two projected LSIs is most concisely seen as follows: Our $\widetilde{I}(\mu\,|\,\rho_*)$ is the squared  $L^2(\mu)$-norm of the projection of $\nabla \log d\mu/d\rho_*$ onto a certain subspace of $L^2(\mu)$, whereas the functional $\overline{I}(\mu\,|\,\rho_*)$ used in \cite{ConfortiLackerPal} is the squared norm of the projection onto the \emph{orthogonal complement} of this subspace. There is a geometric reason for this complementary relationship, explained in Section \ref{se:geometry}.

\subsection{Algorithmic implications for variational inference and related literature} \label{se:VIliterature}

Theorem \ref{th:JKO} suggests that the independent projection \eqref{def:mainSDE} may be a reasonable basis for a Monte Carlo algorithm for sampling from a solution of the mean field variational inference problem \eqref{def:meanfield}, in the spirit of Langevin Monte Carlo.
An implementable algorithm would need to efficiently contend with the McKean-Vlasov nature (i.e., measure-dependence) of the SDE \eqref{def:mainSDE}, in addition to time-discretization. 
We do not attempt here to develop any  algorithms, except to explain why one obvious idea may not be a good one.
This most obvious idea, suggested by the McKean-Vlasov form of the equation, is a particle approximation, in which a large number $m$ particles  are simulated and the $\mu^{-i}$-integral is replaced by an empirical average.
Specifically, consider ``particles" $(\bm{Z}^1,\ldots,\bm{Z}^m)$ each residing in $(\R^d)^n$, so that we might write $\bm{Z}^k=(Z^{ki})_{i=1,\ldots,n}$, with each $Z^{ki}$ residing in $\R^d$. They solve the SDE system  (recalling $\bm{b}=\nabla f$ in this section)
\begin{equation}
dZ^{ki}_t = \frac{1}{m-1}\sum_{j \neq k} \nabla_i f (Z^{j1}_t,\ldots,Z^{j(i-1)}_t,Z^{ki}_t,Z^{j(i+1)}_t,\ldots,Z^{jn}_t)\,dt + \sqrt{2}\,dB^{ik}_t, \label{def:particleapprox}
\end{equation}
for $k=1,\ldots,m$ and $i=1,\ldots,n$, where $(B^{ki})$ are independent Brownian motions, each of dimension $d$. 
Standard results on propagation of chaos show the following convergence. Let $\mu_0 \in \P^{\otimes n}(\R^d)$ be given, and take $(\bm{Z}^k_0)_{k=1,\ldots,m}$ to be iid $\sim\mu_0$. Then, for each fixed $k$, $(\bm{Z}^1,\ldots,\bm{Z}^k)$ converges in law as $m \to \infty$ to the $k$-fold product measure $\mu^{\otimes k}$, where $\mu$ is the law of the solution of \eqref{def:mainSDE} initialized from $\mu_0$. 
This my appear promising, 
as one could then discretize time in the SDE \eqref{def:particleapprox} to achieve an implementable algorithm.
But in order for this to be of practical value, one would want the rate of convergence, with respect to both the time-discretization and particle approximation, to be well-controlled (ideally uniform) with respect to the time horizon.
This poses a serious difficulty.
Non-uniform convergence rates for propagation of chaos tend to deteriorate exponentially with the time horizon. Known uniform-in-time convergence rates mainly rely on a dissipative drift, and the drift of the particle system \eqref{def:particleapprox} does not appear to inherit enough dissipativity from $\nabla f$ even when $f$ is strongly concave.
Lastly, on a conceptual level, the particle approximation \eqref{def:particleapprox} is unappealing because it does not respect the product measure structure of the independent projection; namely, for each $k$, $(Z^{k1},\ldots,Z^{kn})$ are not independent.

Algorithms for variational inference are an active area of research. 
The standard algorithm for mean field variational inference appears to be coordinate ascent, or CAVI \cite[Section 2.4]{blei2017variational}, with the recent work \cite{bhattacharya2023convergence} providing additional references and theoretical guarantees.
Closer in spirit to the present work are the recent papers \cite{yao2022mean,tran2023particle}, which propose discrete-time schemes for mean field variational inference based on a Wasserstein gradient flow perspective. The former adopts a structured latent variable model and the latter in a general setup. Section 4 of \cite{tran2023particle} notably identifies a discrete-time scheme closely related to the Euler approximation for the \emph{particle approximation} of the independent projection, discussed at \eqref{def:particleapprox} above.

Aside from the mean field framework of $\P^{\otimes n}(\R^d)$, common alternative classes of measures for variational inference \eqref{def:meanfield} are exponential families, especially Gaussians.
The recent paper \cite{lambert2022variational}  analyzes a constrained Fokker-Planck equation as a gradient flow for Gaussian (and Gaussian mixture) variational inference and has strong analogies to Section \ref{se:intro:GF} of this paper.

\subsection{Entropic optimality} \label{se:intro:entropic} 

This section makes precise the entropic optimality property announced in Section \ref{se:intro:entropicoptimality} above.
In this section we no longer need to assume that $\bm{b}$ is a gradient.

We start with some notation. 
For a metric space $E$ and for $P \in \P(C(\R_+;E))$ and $t \ge 0$, we write $P[t] \in \P(C([0,t];E))$ for the pushforward of $P$ by the restriction map $x \mapsto x|_{[0,t]}$, and we write $P_t \in \P(E)$ for the pushforward by the evaluation map $x \mapsto x_t$. For $t=0$, noting that $C(\{0\};E) \cong E$, we may identify $P_0=P[0]$.
Define an ``infinitesimal entropy" functional as follows.
For $Q,P\in \P(C(\R_+;E))$ satisfying $H(Q_0\,|\,P_0) < \infty$, let
\begin{align*}
\H'_0(Q\,|\,P) := \liminf_{t \downarrow 0}\frac{H(Q[t]\,|\,P[t]) - H(Q_0\,|\,P_0)}{t}.
\end{align*}
Intuitively, this is the growth rate at time zero of the path-space entropy $t \mapsto H(Q[t]\,|\,P[t])$. 
Note that $\H'_0(Q\,|\,P)$ takes values in $(-\infty,\infty]$. The following is proven in Section \ref{se:proof:entropicoptimality}.

\begin{theorem} \label{th:entropicoptimality}
Let $\rho_0 \in \P_2(\R^{dn})$. Let  $\mu_0 \in \P^{\otimes n}_2(\R^d)$  satisfy $H(\mu_0\,|\,\rho_0) < \infty$.
Let $\rho \in \P(C(\R_+;\R^{dn}))$ denote the law of the solution $\bm{Y}$ of \eqref{def:SDE-reference}, initialized from $\bm{Y}_0 \sim \rho_0$.
Let $\mu \in \P^{\otimes n}(C(\R_+;\R^d))$ denote the law of the solution $\bm{X}$ of \eqref{def:mainSDE}, initialized from $\bm{X}_0 \sim \mu_0$.
Then
\begin{equation}
\H'_0(\mu\,|\,\rho) = \frac14\sum_{i=1}^n \E\big|b^i(\bm{X}_0) - \E[b^i(\bm{X}_0)\,|\,X^i_0]\big|^2 < \infty, \label{eq:entopt-identity}
\end{equation}
and
\begin{equation}
\H'_0(\mu\,|\,\rho) \le \H'_0(\nu\,|\,\rho), \qquad \forall \nu \in \P^{\otimes n}(C(\R_+;\R^{d})) \text{ such that } \nu_0=\mu_0. \label{ineq:entopt}
\end{equation}
\end{theorem}

Note that the optimality property of Theorem \ref{th:entropicoptimality} applies to the McKean-Vlasov equation, which is a special case of the independent projection as in Example \ref{ex:McKeanVlasov}. 
In words, Theorem \ref{th:entropicoptimality} states that the independent projection accumulates path-space relative entropy as slowly as possible, among independent processes.
The proof is given in Section \ref{se:proof:entropicoptimality}, and after a careful application of Girsanov's theorem it essentially boils down to the standard fact that, among $X^i_0$-measurable random variables $Y$, the expectation $\E|Z-Y|^2$ is minimized by $Y=\E[Z|X^i_0]$.
The proof also reveals that equality in \eqref{ineq:entopt} can only hold for those $\nu$ governed by a drift which agrees with that of $\mu$ in an infinitesimal sense near time zero; see Remark \ref{re:equalitycase}.

It is worth highlighting that the optimality property of Theorem \ref{th:entropicoptimality} is purely \emph{dynamic}, in the sense that it does not depend on the relationship between $\mu_0$ and $\rho_0$. Finding a product measure $\mu_0$ which is ``closest" to $\rho_0$ is a separate, \emph{static} problem, as discussed at \eqref{def:meanfield} above.

\begin{remark}
There are, of course, other ways to approximate a given measure $\rho$ by a product measure, which might be optimal for different objectives. For instance, reversing the order of arguments in entropy, the minimizer of $H(\rho \,|\,\cdot)$ over product measures is always uniquely given by the product of the marginals of $\rho$; see \cite[Section 1.3.1]{LackerMukherjeeYeung}. This does not, however, appear to be as useful in applications, where the optimizers of \eqref{def:meanfield} tend to better capture the large-$n$ behavior. Example \ref{ex:McKeanVlasov} is a dramatic illustration of this principle, where the one-particle marginal of the independent projection is exactly equal to the large-$n$ limit of the one-particle marginal of the original dynamics.
\end{remark}

\begin{remark}
The entropic optimality principle in this section is somewhat reminiscent of Dafermos's \emph{entropy rate criterion} \cite{dafermos1973entropy} in the study of hyperbolic PDEs, which was introduced for a rather different purpose, as a criterion for singling out a unique  solution from a multitude. This was also given an interpretation  in terms of gradient flows in \cite{gigli2013entropic}. 
\end{remark}

\subsection{Related work and further examples}
Forms of the independent projection
SDE \eqref{def:mainSDE} and its associated Fokker-Planck equation \eqref{def:newFokkerPlanck} appeared, respectively, in the two recent  papers  \cite{JacksonLacker} and \cite{JabinPoyatoSoler}.
In these papers, the independent projection was given no particular name or rationale, beyond being a mathematically convenient choice of independent process to compare the given dependent process.
Theorem \ref{th:entropicoptimality}, not to mention  the results to follow below, provide some theoretical basis for this choice.

The paper \cite{JacksonLacker} studies stochastic optimal control problems in high dimension, focusing on the optimality gap between \emph{full-information} versus \emph{distributed} controls. Essentially, a full-information control is a collection of functions $(b^i(x^1,\ldots,x^n))_{i=1,\ldots,n}$, and a distributed control requires that the $i$th component $b^i=b^i(x^i)$ depends only on the $i$th variable, for each $i$.
The main results of \cite{JacksonLacker} use the independent projection \eqref{def:mainSDE} as a tool to prove a sharp bound on the optimality gap. We revisit these ideas in Section \ref{se:intro:goodbound} to estimate the (non-infinitesimal) entropy $H(\mu[t]\,|\,\rho[t])$.

The paper \cite{JabinPoyatoSoler} studies mean field limits for interacting particle systems with \emph{heterogeneous} pairwise interactions, of the form
\begin{equation}
dY^i_t = \bigg( K_1(Y^i_t) + \sum_{ j \neq i} A_{ij} K_2(Y^i_t,Y^j_t)\bigg) dt + \sqrt{2} dB^i_t, \label{def:particles}
\end{equation}
where $K_1 : \R^d \to \R^d$ and $K_2 : \R^d \times \R^d \to \R^d$ are given functions, and $A$ is a given $n \times n$ matrix of \emph{interaction weights}, which we take to be zero on the diagonal for simplicity. Of course, this is just the SDE \eqref{def:SDE-reference} with the following specification for $\bm{b}=(b^1,\ldots,b^n)$:
\begin{equation}
b^i(\bm{x}) = K_1(x^i) + \sum_{j \neq i} A_{ij} K_2(x^i,x^j). \label{def:pairwiseinteractions}
\end{equation}
The most common special case is when $A_{ij}=1/(n-1)$ for all $i \neq j$, for which \eqref{def:SDE-reference} becomes the classical symmetric model \eqref{def:meanfieldparticles}.
For very general weights $(A_{ij})$, the paper \cite{JabinPoyatoSoler} approximates the (marginals of the) particle system \eqref{def:particles} by its independent projection, and the latter is then used to identify new large-$n$ limits which generalize the McKean-Vlasov equation \eqref{def:McKeanVlasov}, depending on a certain graphon-type limit for $A$.

\begin{remark} \label{re:rowsums}
The discussion of Example \ref{ex:McKeanVlasov} generalizes, as will be shown in Proposition \ref{pr:McKeanVlasov}.
Suppose that the sum along each row of $A$ equals $1$. That is, $A$ is the transition matrix of some Markov chain on $\{1,\ldots,n\}$.  Then, for iid initial positions, it is not difficult to see that the independent projection of \eqref{def:particles} is once again given by $n$ independent copies of the McKean-Vlasov equation \eqref{def:McKeanVlasov}. This was observed already in \cite[Remark 2.3]{JabinPoyatoSoler}.
\end{remark}

\begin{example}\label{ex:lineardrift}
As a special case of \eqref{def:particles}, the case of \emph{linear} drift is, unsurprisingly, explicitly solvable and may confer some intuition. Consider, for example, $d=1$ and $\bm{b}(\bm{x})= A\bm{x}$ for some $n \times n$ matrix $A=(A_{ij})$. Then \eqref{def:mainSDE} becomes
\begin{align*}
dX^i_t = \Big( A_{ii} X^i_t + \sum_{j \neq i} A_{ij}  \E[X^j_t]\Big)\,dt + \sqrt{2}\,dB^i_t.
\end{align*}
Taking expectations, we have $(d/dt)\E\bm{X}_t=A\E\bm{X}_t$. If $\E\bm{X}_0=0$, then $\E\bm{X}_t=0$ for all $t$, and so
\begin{align*}
dX^i_t = A_{ii}X^i_t dt + \sqrt{2}\,dB^i_t.
\end{align*}
In other words, the independent projection simply zeros out the interaction terms.
\end{example}

\subsection{On the proximity of the independent projection} \label{se:intro:goodbound}

We have now discussed the main points of this paper, senses in which $\bm{X}$ is the best approximation of $\bm{Y}$ by independent processes. But to what extent is it a good approximation? This section gives two answers. The first is quantified in terms of relative entropy on path space which is natural in light of Theorem \ref{th:entropicoptimality}, and the second works instead with time-marginals. We describe also an application to quantitative mean field limits for non-exchangeable interacting particle systems of the form \eqref{def:particles}. The approach is based on functional inequalities and is inspired by the author's prior works \cite{JacksonLacker} and \cite{LackerMukherjeeYeung}.

First, given a time horizon $T > 0$, let us explain why $H(\mu[T]\,|\,\rho[T]) = o(n)$ is the right definition of a ``good approximation" in this context.
For $k \le n$, let $V_k$ denote the set of vectors $v=(v_1,\ldots,v_k)$ of distinct elements of $\{1,\ldots,n\}$. For $v \in V_k$, let $\mu^v$ denote the marginal law of $(X^{v_i})_{i=1,\ldots,k}$ under $\bm{X} \sim \mu$, and define $\rho^v$ similarly.
The Wasserstein distance obeys a well known subadditivity  property \cite[Section 3.4]{LackerMukherjeeYeung} when one of its arguments is a product measure:
\begin{equation}
\W_{(k)}^2(\mu[T],\rho[T]) := \frac{1}{|V_k|}\sum_{v \in V_k}\W_2^2(\mu^v[T],\rho^v[T]) \le \frac{2k}{n}\W_2^2(\mu[T],\rho[T]).  \label{def:subadditivity}
\end{equation}
Because $\bm{b}$ is Lipschitz, $\rho[T]$ obeys a Talagrand inequality as soon as its initial law $\rho_0$ does \cite[Proposition C.1]{lacker2023hierarchies}, which entails that
\begin{equation}
\W_2^2(\mu[T],\rho[T]) \le CH(\mu[T]\,|\,\rho[T]), \label{def:TI}
\end{equation}
for some constant $C$ depending on $T$ and the Lipschitz constant of $\bm{b}$. Combining the previous two inequalities, we see for fixed $k$ that  $\W_{(k)}^2(\mu[T],\rho[T]) =o(1)$ as long as $H(\mu[T]\,|\,\rho[T]) = o(n)$. In other words, if $H(\mu[T]\,|\,\rho[T]) = o(n)$, then we deduce that most low-dimensional marginals of $\mu$ and $\rho$ are close to each other.
This further implies a concentration bound for the empirical measure, by following the reasoning of \cite[Corollary 1.2]{LackerMukherjeeYeung}.

Having now motivated the goal of showing $H(\mu[T]\,|\,\rho[T]) = o(n)$, we state a result in this direction which resembles the main results of \cite{LackerMukherjeeYeung,JacksonLacker}. In the following, we say that a probability measure $m$ on a Euclidean space satisfies a Poincar\'e inequality with constant $c$ if, for all differentiable $f \in L^2(m)$,
\begin{equation*}
\int f^2\,dm - \Big(\int f\,dm\Big)^2 \le c\int|\nabla f|^2\,dm.
\end{equation*}

\begin{theorem} \label{th:proximity}
Let $L$ be the Lipschitz constant of $\bm{b}$, and let $T > 0$. Let $\rho_0 \in \P_2(\R^{dn})$ and $\mu_0 \in \P^{\otimes n}_2(\R^d)$.
Suppose $\mu_0$ satisfies a Poincar\'e inequality with constant $c_0$. Define $c_t := c_0 e^{2L^2 t} +(e^{2L^2t}-1)/L^2$, then
\begin{equation}
H(\mu[T]\,|\,\rho[T]) \le H(\mu_0\,|\,\rho_0) + \frac14\int_0^Tc_t\sum_{i=1}^n\sum_{j=1,\,j\neq i}^n \E\|\nabla_j b^i(\bm{X}_t)\|_{\mathrm{Frob}}^2\,dt. \label{ineq:proximity1}
\end{equation}
\end{theorem}

Note that the right-hand side of \eqref{ineq:proximity1} vanishes when $\rho_0=\mu_0$ and $b^i(\bm{x})=b^i(x^i)$ is a function of $x^i$ only. In this case, $\mu=\rho$. There is a time-uniform version as well, in the log-concave gradient flow regime:

\begin{theorem}\label{th:proximity-uniform}
Suppose $\bm{b}=\nabla f$ for a $\kappa$-concave function $f$ with Lipschitz gradient, where $\kappa > 0$. Suppose $\rho_0 \in \P_2(\R^{dn})$ satisfies the LSI with constant $\eta_0$, in the sense of \eqref{def:LSI}, and suppose $\mu_0 \in \P^{\otimes n}_2(\R^d)$ satisfies the Poincar\'e inequality with constant $c_0$. Define $\eta=\min(\kappa,\eta_0)$ and $c=\max(c_0,1/\kappa)$. 
Then, for all $t \ge 0$,
\begin{equation*}
H(\mu_t\,|\,\rho_t) \le e^{-  \eta t }H(\mu_0\,|\,\rho_0) + \frac{c}{2}\sum_{i=1}^n\sum_{j=1, \, j \neq i}^n\int_0^te^{-   \eta (t-s) }\E \|\nabla_{ij} f(\bm{X}_s)\|_{\mathrm{Frob}}^2 \,ds. 
\end{equation*}
\end{theorem}

Notably, applying Theorem \ref{th:proximity-uniform} to stationary solutions recovers an estimate on the invariant measure $\mu_*$, which is given as the unique optimizer for \eqref{def:meanfield} guaranteed by Theorem \ref{th:convex-limit}(3). Indeed, consider the stationary solutions $\mu_t=\mu_*$ and $\rho_t=\rho_*$, and note that both measures are $\kappa$-log-concave so that $c=1/\kappa$ and $\eta=\kappa$. Sending $t\to\infty$ in the bound of Theorem \ref{th:proximity-uniform} yields
\begin{align*}
\inf_{\mu \in \P^{\otimes n}(\R^d)}H(\mu\,|\,\rho_*) = H(\mu_*\,|\,\rho_*) \le \frac{1}{2\kappa^2 }\sum_{i=1}^n\sum_{j=1, \, j \neq i}^n\int \|\nabla_{ij} f\|_{\mathrm{Frob}}^2\,d\mu_*.
\end{align*}
This recovers the main estimate of\cite[Theorem 1.1]{LackerMukherjeeYeung}, even with the same constant.

The pairwise interactions described in \eqref{def:pairwiseinteractions} provide a more interesting class of examples for which the bounds of Theorems \ref{th:proximity} and \ref{th:proximity-uniform} behave well. For $\bm{b}$ as in \eqref{def:pairwiseinteractions},
\begin{align*}
\sum_{i=1}^n\sum_{j=1,\,j\neq i}^n \E\|\nabla_j b^i(\bm{X}_t)\|_{\mathrm{Frob}}^2 &\le \|\|\nabla_2 K_1\|_{\mathrm{Frob}}\|_\infty^2\sum_{i=1}^n\sum_{j=1,\,j\neq i}^n  A_{ij}^2 = \|\|\nabla_2 K_1\|_{\mathrm{Frob}}\|_\infty^2\tr(AA^\top),  
\end{align*}
where $\nabla_2 K_2(x^1,x^2)$ is the gradient with respect to $x^2$. This bound on \eqref{ineq:proximity1} in combination with \eqref{def:TI} and \eqref{def:subadditivity} leads immediately to the following corollary, after we recall from Remark \ref{re:rowsums} that the independent projection is given by $n$ iid copies of the McKean-Vlasov equation if the sum along each row of $A$ is equal to 1.

\begin{corollary} \label{co:propchaos}
Suppose $\bm{b}$ is given by \eqref{def:pairwiseinteractions}, where $K_1$ and $K_2$ are Lipschitz, and where the matrix $A$ has diagonal entries equal to zero and row sums equal to 1. Let $T > 0$, and let $\overline{\mu}_0 \in \P_2(\R^d)$ satisfy a Poincar\'e inequality. Let $\rho[T] \in \P(C([0,T];\R^{dn}))$ be the law of the solution of \eqref{def:particles} initialized from $\rho_0=\overline{\mu}_0^{\otimes n}$, and let $\overline{\mu}[T] \in \P(C([0,T];\R^{d}))$ be the law of the solution of the McKean-Vlasov equation \eqref{def:McKeanVlasov} initialized from $\overline{\mu}_0$. Then we have
\begin{equation}
\W_{(k)}^2(\overline{\mu}^{\otimes n}[T],\rho[T]) \le C (k/n)\tr(AA^\top), \label{ineq:propchaos}
\end{equation}
where $\W_{(k)}^2(\overline{\mu}^{\otimes n}[T],\rho[T])$ was defined in \eqref{def:subadditivity}, and where the constant $C$ depends only on $T$, the Lipschitz constants of $(K_1,K_2)$, and the Poincar\'e constant of $\overline{\mu}_0$. 
\end{corollary}


The bound  of Corollary \ref{co:propchaos} is useful under the ``mean field condition" $\tr(AA^\top) = o(n)$, which appeared also in \cite{JacksonLacker,LackerMukherjeeYeung} as well as prior work \cite{BasakMukherjee} on the Ising and Potts models with general interaction weights.
This condition covers a wide range of situations. For example, consider the transition matrix of a simple random walk on a graph: Let $A_{ij}^n=(1/d^n_i)1_{i \sim j}$, for some graph on vertex set $\{1,\ldots,n\}$ with vertex $i$ having degree $d^n_i > 0$, where $i \sim j$ indicates adjacency. Then $\tr(A^n(A^n)^\top) = \sum_{i=1}^n(1/d^n_i)$, which is $o(n)$ as long as the graph sequence is ``dense" in the rather weak sense that the degree of a uniformly randomly chosen vertex is divergent in probability. 

We may interpret Corollary \ref{co:propchaos} as a \emph{universality} result, in the sense that it identifies a broad class of interaction matrices $A$  for which the particle system \eqref{def:particles} is asymptotically characterized by the same McKean-Vlasov equation. Universality results of this kind for interacting diffusions have been the subject of a growing body of literature. We refer to \cite{oliveira2019interacting} for a  qualitative result of this nature, as well as \cite{bris2022note} for a quantitative uniform-in-time version, and to both papers for a more extensive discussion of this literature.

Let us lastly state a uniform-in-time version of Corollary \ref{co:propchaos}, which arises from applying Theorem \ref{th:proximity-uniform} instead of \ref{th:proximity}. It applies to the case where $\bm{b}=\nabla f$, where $f$ takes the form
\begin{equation}
f(\bm{x}) = \sum_{i=1}^n U(x^i) + \frac12\sum_{i,j=1}^n A_{ij} V(x^i-x^j), \label{def:heterogeneousGibbs}
\end{equation}
Then, if $A$ is symmetric and $V$ is even, $\bm{b}$ takes the form of \eqref{def:pairwiseinteractions}, with $K_1(x)=\nabla U(x)$ and $K_2(x,y)=\nabla V(x-y)$.

\begin{corollary} \label{co:propchaos-uniform}
Suppose $\bm{b}=\nabla f$ with $f$ given by \eqref{def:heterogeneousGibbs}. Assume $U,V : \R^d \to \R$ are concave with Lipschitz gradient, with $U$ being $\kappa$-concave and $V$ being even. Assume $A$ is symmetric, with nonnegative entries, diagonal entries equal to zero, and row sums equal to 1. Let $\rho_t$ be the law of $Y_t$ which satisfies \eqref{def:SDE-Langevin} initialized from $Y_0 \sim \rho_0 \in \P((\R^d)^n)$ satisfying the LSI with constant $\eta_0$ in the sense of \eqref{def:LSI}. Let $\overline{\mu}_0 \in \P_2(\R^d)$ satisfy a Poincar\'e inequality with constant $c_0$, and let $\overline{\mu}_t \in \P(\R^{d})$ be the time-$t$ law of the solution of the McKean-Vlasov equation \eqref{def:McKeanVlasov} initialized from $\overline{\mu}_0$, again with $K_1(x)=\nabla U(x)$ and $K_2(x,y)=\nabla V(x-y)$. Then, with $\eta=\min(\kappa,\eta_0)$ and $c=\max(c,1/\kappa)$, we have for all $t \ge 0$
\begin{align*}
\W_{(k)}^2(\overline{\mu}^{\otimes n}_t,\rho_t) &\le \frac{4k}{\eta n} H(\overline{\mu}^{\otimes n}_t\,|\,\rho_t),  \quad \text{and } \\
H(\overline{\mu}^{\otimes n}_t\,|\,\rho_t) &\le e^{-  \eta t }H(\overline{\mu}^{\otimes n}_0\,|\,\rho_0) + \frac{c}{2\eta}(1-e^{-\eta t})   \|\|\nabla^2 V\|_{\mathrm{Frob}}\|_\infty^2 \tr(A^2). 
\end{align*} 
\end{corollary}

Again we highlight similarities with the recent paper \cite{bris2022note}, which proved quantitative uniform-in-time propagation of chaos for interacting diffusions on graphs. They quantified the convergence in terms of the expected distance of the empirical measure, and thus their rate is not directly comparable to ours. They do not require (global) convexity like we do, but the dependence of their estimate on the graph (via $A$, in our bound) appears to be less sharp.

\subsection{Outline of the rest of the paper}

The rest of the paper gives the proofs of the above results, with the exception of Section \ref{se:geometry} which discusses the formal Wasserstein geometry of $\P^{\otimes n}_2(\R^d)$. Section \ref{se:wellposedness} discusses some preliminary results on well-posedness and moment bounds for the independent projection. Then, Section \ref{se:proof:JKO} proves Theorem \ref{th:JKO} on the JKO scheme. Section \ref{se:proof:longtime} proves the results on long-time behavior, and Section \ref{se:proof:entropicoptimality} proves Theorem \ref{th:entropicoptimality} on entropic optimality.

\subsection*{Acknowledgement}

The author is indebted to Sam Power for helpful discussions about variational inference and for pointing out the references \cite{tran2023particle,yao2022mean}.

\section{The geometry of $\P_2^{\otimes n}(\R^d)$} \label{se:geometry}

The famous work of Otto \cite{otto2001geometry}  described the formal Riemannian structure of the Wasserstein space $(\P_2(\R^{dn}),\W_2)$, in which one can view the Langevin dynamics (or rather its associated Fokker-Planck equation) as a gradient flow directly  in continuous-time rather than through the discrete-time JKO scheme.
In this section we describe how the Otto calculus adapts to the submanifold $\P_2^{\otimes n}(\R^d)$ and how to view  the independent projection as a gradient flow on this submanifold.
Throughout, we adopt the notation $\rho_*(dx) \propto e^{f(x)}dx$ and $\bm{b}=\nabla f$ as in Section \ref{se:intro:GF}.

We will not discuss the full derivation of Otto calculus, but we rather mention some relevant highlights; we refer also to the classic book \cite{AGSbook} for a rigorous treatment of gradient flows on Wasserstein space. The tangent space of $\P_2(\R^{dn})$ at $\rho$ is given by
\begin{equation}
\mathrm{Tan}_\rho\P_2(\R^{dn}) = \overline{ \{\nabla\varphi : \varphi \in C^\infty_c(\R^{dn})\} }^{L^2(\rho;\R^{dn})}. \label{def:tangentspace}
\end{equation}
The absolutely continuous curves $(\rho_t)$ in $\P_2(\R^{dn})$ are those satisfying (weakly) the continuity equation $\partial_t \rho_t + \nabla \cdot (\rho_t v_t) = 0$, with $v_t$ being a uniquely determined element of $\mathrm{Tan}_{\rho_t}\P_2(\R^{dn})$ for a.e.\ $t$, which we consider to be the ``velocity" of the curve at time $t$.
The so-called \emph{Wasserstein gradient} of the functional  $H(\cdot\,|\,\rho_*)$ at a probability density $\rho$ is given by $\nabla \log(\rho/\rho_*)$. 
The Fokker-Planck equation  \eqref{def:FokkerPlanck} can be written (for $\bm{b}=\nabla f=\nabla\log\rho_*$)
\begin{equation}
\partial_t\rho_t - \mathrm{div}\big(\rho_t \nabla \log(\rho_t/\rho_*)\big) = 0. \label{def:FokkerPlanckGF}
\end{equation}
This says that the velocity of $\rho_t$ is the negative Wasserstein gradient of $H(\cdot\,|\,\rho_*)$ at $\rho_t$,  expressing the fact that \eqref{def:FokkerPlanckGF} describes the gradient flow of $H(\cdot\,|\,\rho_*)$ in $\P_2(\R^{dn})$.

Before turning to the submanifold $\P^{\otimes n}_2(\R^d)$, let us first recall a basic principle for finite-dimensional Riemannian submanifolds. For an embedded submanifold $M$ of a Euclidean space $\R^k$ and a smooth function $h : \R^k \to \R$, the restriction $h|_M$ defines a smooth function on $M$, and the gradient (in the geometry of $M$) at a point $x \in M$ is simply the orthogonal projection of the Euclidean gradient $\nabla h(x)$ onto the tangent space of $M$ at $x$.

Hence, to understand gradient flows in the submanifold $\P^{\otimes n}_2(\R^d) \subset \P_2(\R^{dn})$, we need to identify its tangent spaces.
We claim that, for $\mu = \mu^1 \otimes \cdots \otimes \mu^n \in \P_2^{\otimes n}(\R^d)$, we should make the identification
\begin{equation}
\mathrm{Tan}_\mu\P_2^{\otimes n}(\R^d) = \bigoplus_{i=1}^n\mathrm{Tan}_{\mu^i}\P_2(\R^d). \label{def:tangentspace-new}
\end{equation}
The simplest way to understand this is to realize that $\P^{\otimes n}_2(\R^d)$ is diffeomorphic to $(\P_2(\R^d))^n$ via the obvious bijection $\mu^1\otimes \cdots \otimes \mu^n \mapsto (\mu^1,\ldots,\mu^n)$, where the space $(\P_2(\R^d))^n$ is given the geometry of the $n$-fold product of the Riemannian manifold $\P_2(\R^d)$.
More concretely, the right-hand side of \eqref{def:tangentspace-new} is the space of functions of the form $(x^1,\ldots,x^n) \mapsto (v_1(x^1),\ldots,v_n(x^n))$, where $v_i \in \mathrm{Tan}_{\mu^i}\P_2(\R^d)$ for each $i=1,\ldots,n$. Another way to see \eqref{def:tangentspace-new} is as the closure of the space of gradients of smooth \emph{additively separable} functions:
For real-valued functions $\varphi_1,\ldots,\varphi_n$ on $\R^d$, write $\varphi_1\oplus \cdots \oplus \varphi_n$ for the function on $\R^{dn}$ given by
\begin{equation*}
(\varphi_1\oplus \cdots \oplus \varphi_n)(x^1,\ldots,x^n) = \sum_{i=1}^n\varphi_i(x^i).
\end{equation*}
Then \eqref{def:tangentspace-new} can be written as
\begin{equation*}
\mathrm{Tan}_\mu\P_2^{\otimes n}(\R^d) = \overline{ \{\nabla(\varphi_1\oplus\cdots\oplus\varphi_n) : \varphi_i \in C^\infty_c(\R^d), \ i=1,\ldots,n\}}^{L^2(\mu;\R^{dn})}.
\end{equation*}

With this identification of the tangent spaces, we can return to the topic of gradient flows.
For $\mu \in \P^{\otimes n}_2(\R^d)$, let $\mathsf{T}_{\mu}$ denote the $L^2(\mu;\R^{dn})$-projection onto  $\mathrm{Tan}_\mu\P_2^{\otimes n}(\R^d)$. The action of this projection on $\mathrm{Tan}_\mu\P_2(\R^{dn})$ is not difficult to compute. We claim for $g \in C^\infty_c(\R^{dn})$ that
\begin{equation}
\mathsf{T}_\mu \nabla g =\big(
\E_\mu[\nabla_1 g(\bm{X})\,|\, X^1] ,\ldots, \E_\mu[\nabla_n g(\bm{X})\,|\, X^n]\big),  \label{def:Tprojection}
\end{equation}
where $\bm{X}\sim \mu$ in the conditional expectations. Indeed, we have for any $\varphi_1,\ldots,\varphi_n \in C^\infty_c(\R^d)$ that
\begin{align}
\int_{\R^{dn}} \nabla g \cdot \nabla (\varphi_1\oplus\cdots\oplus \varphi_n) \,d\mu &= \sum_{i=1}^n \int_{\R^{dn}}\nabla_i g(\bm{x}) \cdot \nabla_i\varphi_i(x^i)\,\mu(d\bm{x}) \nonumber \\
	&= \sum_{i=1}^n\int_{\R^d}\E_\mu[\nabla_i g(\bm{X}) \,|\,X^i=x^i] \cdot \nabla_i\varphi_i(x^i)\,\mu^i(dx^i). \label{def:proj-calculuation}
\end{align}
Because $\mu$ is a product measure and $g \in C^\infty_c(\R^{dn})$, we may exchange the derivative and expectation, $\nabla_i \E_\mu[  g(\bm{X})\,|\, X^i=x^i] = \E_\mu[\nabla_i g(\bm{X})\,|\, X^i=x^i]$. It follows that the right-hand side of \eqref{def:Tprojection} is indeed an element of $\mathrm{Tan}_\mu\P_2(\R^{dn})$, and then the formula \eqref{def:proj-calculuation} implies the claim \eqref{def:Tprojection}. Taking $L^2$-limits in \eqref{def:Tprojection} extends the formula beyond smooth $\nabla g$, and we find for any $v=(v_1,\ldots,v_n) \in \mathrm{Tan}_\mu\P_2(\R^{dn})$ that
\begin{equation*}
\mathsf{T}_\mu v =\big(
\E_\mu[v_1(\bm{X})\,|\, X^1] ,\ldots, \E_\mu[v_n(\bm{X})\,|\, X^n]\big). 
\end{equation*}

Finally, let us return the PDE \eqref{def:newFokkerPlanck} related to our independent projection, again in the case that $\bm{b}=\nabla f=\nabla\log\rho_*$.
For a sufficiently regular product measure $\mu$ we have $\mathsf{T}_\mu\nabla \log\mu = \nabla \log \mu$, the latter already belonging to $\mathrm{Tan}_\mu\P^{\otimes n}_2(\R^d)$ because $\log\mu$ is additively separable.
Thus, applying \eqref{def:Tprojection} with $g=f=\log\rho_*$, the PDE \eqref{def:newFokkerPlanck} can be written as
\begin{equation}
\partial_t\mu_t  -\mathrm{div}\big(\mu_t \mathsf{T}_{\mu_t}\nabla \log(\mu_t/\rho_*)\big) = 0. \label{def:PDE-GF-interpretation}
\end{equation}
This is analogous to the formulation \eqref{def:FokkerPlanckGF} of the Fokker Planck equation, stating that the velocity of $\mu_t$ is the projection of the negative Wasserstein gradient of $H(\cdot\,|\,\rho_*)$ at $\mu_t$ onto the tangent space  $\mathrm{Tan}_{\mu_t}\P_2^{\otimes n}(\R^d)$. This expresses the fact that \eqref{def:PDE-GF-interpretation} is the gradient flow of $H(\cdot\,|\,\rho_*)$ in $\P_2^{\otimes n}(\R^d)$, because the projection of the gradient is precisely the gradient within the submanifold geometry.

This completes our discussion of the gradient flow property of the independent projection, and the rest of this section documents some further observations about geometry of $\P_2(\R^{dn})$. Consider the ``marginal map" $\pi : \P_2(\R^{dn}) \to \P_2^{\otimes n}(\R^d)$, which maps a measure to the product of its marginals. A level set might be written as
\[
\pi^{-1}(\{\mu^1 \otimes \ldots \otimes \mu^n\}) = \Pi(\mu^1,\ldots,\mu^n),
\]
the latter being the standard notation for the set of probability measures on $\R^{dn} \cong (\R^d)^n$ with marginals $\mu^1,\ldots,\mu^n$. We can decompose $\P_2(\R^{dn})$ into the disjoint union of these level sets,
\[
\P_2(\R^{dn}) = \bigcup_{\mu^1\otimes \cdots \otimes \mu^n \in  \P_2^{\otimes n}(\R^d)} \Pi(\mu^1,\ldots,\mu^n).
\]
It was argued in \cite{ConfortiLackerPal} (in the case of $d=2$, but the argument obviously extends) that this marginal map $\pi$ should be viewed as a  Riemannian submersion, as its differential $D\pi(\mu)$ at any point $\mu\in\P_2(\R^{dn})$ is the orthogonal projection
\begin{align*}
D\pi(\mu) : \mathrm{Tan}_\mu\P_2(\R^{dn}) \to \bigoplus_{i=1}^n \mathrm{Tan}_{\mu^i}\P_2(\R^d)=\mathrm{Tan}_{\pi(\mu)}\P_2^{\otimes n}(\R^d),
\end{align*}
with the latter identity being the formula  \eqref{def:tangentspace-new}. We may thus interpret $\mathrm{Tan}_{\pi(\mu)}\P_2^{\otimes n}(\R^d)$ as the \emph{horizontal space} at $\mu$ of the  submersion $\pi$.
The \emph{vertical space} is the orthogonal complement of the horiztonal space in $\mathrm{Tan}_\mu\P_2(\R^{dn})$, which is precisely $\mathrm{Tan}_{\mu}\Pi(\mu^1,\ldots,\mu^n)$, the tangent space of the submanifold of couplings  which played a central role in  \cite{ConfortiLackerPal}.

\begin{remark}
An analogous complementary relationship appeared in the work of Otto \cite[Page 133]{otto2001geometry}: Let $\lambda$ denote Lebesgue measure on the torus $\T^d$, and let $\mathcal{M}$ be the space of diffeomorphisms of $\T^d$, equipped with the $L^2(\lambda)$ inner product. Define $G : \mathcal{M} \to \P(\T^d)$ by $G(\varphi) := \varphi_{\#}\lambda$. Otto (formally) argued that $G$ is a Riemannian submersion and derived the Wasserstein geometry (as it is now known) of $\P(\T^d)=G(\mathcal{M})$ as the geometry induced via $G$ by the flat space $\mathcal{M}$. The complementary geometry here is that of the group of measure-preserving diffeomorphisms  $G^{-1}(\{\lambda\})$, which has been studied extensively since Arnold \cite{arnold1966geometrie} reinterpreted the (incompressible inviscid) Euler equations as the geodesic equation on $G^{-1}(\{\lambda\})$.
\end{remark}

The above discussion reveals a close geometric connection between the space of product measures and spaces of couplings.
In the context of entropic optimal transport, the recent paper \cite{ConfortiLackerPal} introduced and studied 
an analogous SDE in which a space of couplings is preserved, rather than the space of product measures as in the present paper. Their SDE is just like \eqref{def:mainSDE-probabilistic}, except that $\E[b^i(\bm{X}_t)\,|\,X^i_t]$ is replaced by $b^i(\bm{X}_t) - \E[b^i(\bm{X}_t)\,|\,X^i_t]$, which reflects the fact discussed above that the tangent spaces of product measures and couplings are complementary.
Despite the close parallels, the SDE of this paper behaves quite different from that of \cite{ConfortiLackerPal}, in several ways. Most obviously, unlike in our setting, the invariant measure (energy minimizer) is always unique in \cite{ConfortiLackerPal}. Yet, remarkably, our SDE turns out to be significantly simpler to analyze. On a technical level, this is because the equation of \cite{ConfortiLackerPal} depends on \emph{conditional} measures.
On a conceptual level, a key difference is in the structure of the constraint sets:
\begin{itemize}
\item $\Pi(\mu^1,\ldots,\mu^n)$ is convex but not geodesically convex.
\item $\P^{\otimes n}_2(\R^d)$ is geodesically convex but not convex.
\end{itemize}
To be clear, the word ``convex"  (without the modifier ``geodesically") is meant in the usual sense, for subsets of the vector space of signed measures on $\R^{dn}$ of finite variation.
Convexity of $\Pi(\mu^1,\ldots,\mu^n)$ and non-convexity of $\P^{\otimes n}_2(\R^d)$ are easy to see.
Geodesic non-convexity of $\Pi(\mu^1,\ldots,\mu^n)$ is shown in \cite[Proposition 2.1]{ConfortiLackerPal}.
Below we give two proofs of geodesic convexity of $\P^{\otimes n}_2(\R^d)$:

\begin{proposition}
The set of absolutely continuous (with respect to Lebesgue measure) elements of $\P_2^{\otimes n}(\R^d) \subset \P_2(\R^{dn})$ is geodesically convex.
\end{proposition}
\begin{proof}
Let $\mu,\nu \in \P_2^{\otimes n}(\R^d)$ be  absolutely continuous. Let $\mu^i$ and $\nu^i$ denote the marginals. Let $\varphi_i : \R^d \to \R$ be the Brenier map, i.e., the unique (up to additive constants) convex function such that $\nabla\varphi_i$ pushes $\mu^i$ forward to $\nu^i$. Then $\varphi(x^1,\ldots,x^n) = \sum_i\varphi_i(x^i)$ is convex, and $\nabla\varphi$ pushes $\mu$ forward to $\nu$. By uniqueness of the Brenier map, we deduce that the geodesic from $\mu$ to $\nu$ is given by $\mu_t =  ((1-t)\mathrm{Id}+t\nabla\varphi)_{\#}\mu$, which is easily seen to be a product measure for each $t$.
\end{proof}

Without the absolute continuity restriction, somewhat more care is required:

\begin{proposition}
Let $\mu_0,\mu_1 \in \P_2^{\otimes n}(\R^d)$. Then there exists a geodesic in $(\P_2(\R^{dn}),\W_2)$ which connects $\mu_0$ and $\mu_1$ and lies entirely in $\P_2^{\otimes n}(\R^d)$.
\end{proposition}
\begin{proof}
By definition,
\begin{equation*}
\W_2^2(\mu_0,\mu_1) = \inf \sum_{i=1}^n\E|X^i_0-X^i_1|^2,
\end{equation*}
where the infimum is over all  $\bm{X}_0 \sim \mu$ and $\bm{X}_1 \sim \nu$. The infimum does not change if one restricts to the case where $(X^1_0,X^1_1)$ and $(X^2_0,X^2_1)$ and $\cdots$ and $(X^n_0,X^n_1)$ are independent. Hence, there exists an optimal coupling $(\bm{X}_0,\bm{X}_1)$ satisfying this independence property.
Thanks to this independence, the law $\mu_t$ of $\bm{X}_t := t\bm{X}_1 + (1-t)\bm{X}_0$ is a product measure. Because $(\bm{X}_0,\bm{X}_1)$ are optimally coupled, $\mu_t$ defines a geodesic in $(\P_2(\R^{dn}),\W_2)$ by \cite[Theorem 7.2.2]{AGSbook}.
\end{proof}

\begin{remark} \label{se:AGSbook}
Using the geodesic convexity of $\P_2^{\otimes n}(\R^d)$, along with the geodesic (strong) convexity of $H(\cdot\,|\,\rho_*)$ in the case of (strongly) concave $f$, we expect that the most powerful results should apply from the general theory of gradient flows on metric spaces \cite{AGSbook}. Indeed, this should immediately lead to analogues of Theorem \ref{th:JKO}, \ref{th:convex-limit}(iii), and Theorem \ref{th:LSI} but with an abstract notion of gradient flow taking the places of the (PDE associated to the) SDE \eqref{def:mainSDE}. But it is not clear that this approach would save any effort, as it is a non-trivial task to rigorously connect the abstract gradient flow to our SDE. (Though the formal connection via Otto calculus is clear, as explained in this section.)
\end{remark}

\section{Preliminaries: Existence, uniqueness, and estimates} \label{se:wellposedness}

We first collect some fairly straightforward results on the well-posedness and integrability of the independent projection, as defined by the SDE \eqref{def:mainSDE}. Throughout, we assume that $\bm{b}$ is Lipschitz.

\begin{proposition} \label{pr:wellposed}
The SDE \eqref{def:mainSDE} admits a unique strong solution, for any independent and integrable initial positions  $(X^1_0,\ldots,X^n_0)=\bm{X}_0$. If $p \ge 1$ and $\E|\bm{X}_0|^p < \infty$, then $\sup_{t \in [0,T]}\E|\bm{X}_t|^p < \infty$ for all $T > 0$.
\end{proposition}
\begin{proof}
Existence and uniqueness follows as in \cite[Theorem 1.1]{sznitman1991topics}, which was stated for bounded Lipschitz drift but adapts easily to the unbounded (but still Lipschitz) case as long as the initial condition is integrable.
The integrability claim follows from standard arguments using the Lipschitz assumption and Gr\"onwall's inequality. 
\end{proof}

The next proposition justifies the important claims made in Example \ref{ex:McKeanVlasov} and Remark \ref{re:rowsums}, that in certain cases the independent projection reduces to iid copies of the McKean-Vlasov equation.

\begin{proposition} \label{pr:McKeanVlasov}
Suppose $X^1_0,\ldots,X^n_0$ are iid and integrable with law $\mu_0 \in \P(\R^d)$.
Suppose $\bm{b}=(b^1,\ldots,b^n)$ is given by
\[
b^i(\bm{x}) = K_1(x^i) + \sum_{j \neq i} A_{ij} K_2(x^i,x^j).
\]
for Lipschitz functions $K_1$  and $K_2$, and for any matrix $A$ with row sums $\sum_j A_{ij}=1$ for all $i$.
Then the law of independent projection defined in \eqref{def:mainSDE} is given by $\overline{\mu}^{\otimes n}$, where $\overline{\mu}$ is the law of  the unique solution of the McKean-Vlasov equation \eqref{def:McKeanVlasov}.
\end{proposition}
\begin{proof}
On a common probability space, construct iid copies $Z^1,\ldots,Z^n \sim \overline{\mu}$ of the McKean-Vlasov equation:
\begin{align*}
dZ^i_t = \bigg(K_1(Z^i_t) + \int_{\R^d} K_2(Z^i_t,x)\,\overline{\mu}_t(dx)\bigg)dt + \sqrt{2}\,dB^i_t,
\end{align*}
where $B^1,\ldots,B^n$ are independent Brownian motions, and the initializations $Z^i_0$ are iid $\sim \mu_0$. 
Let $\mu_t=\overline{\mu}_t^{\otimes n}$ denote the law of  $(Z^1_t,\ldots,Z^n_t)$.
The point is just to notice that the row sum condition entails
\begin{align*}
\int_{(\R^d)^{n-1}} b^i(Z^i_t,\bm{x}^{-i})\,\mu_t^{-i}(d\bm{x}^{-i}) &= K_1(Z^i_t) +\sum_{j\neq i} A_{ij} \int_{\R^d} K_2(Z^i_t,x)\,\overline{\mu}_t(dx) \\
	&= K_1(Z^i_t) + \int_{\R^d} K_2(Z^i_t,x)\,\overline{\mu}_t(dx).
\end{align*}
Thus, $Z^1,\ldots,Z^n$ solves the independent projection SDE \eqref{def:mainSDE}, which is unique by Proposition \ref{pr:wellposed}.
\end{proof}

\begin{lemma} \label{le:momentbound}
Suppose there exist $c_1 \ge 0$ and $c_2 > 0$ such that
\begin{equation}
\bm{x}\cdot \bm{b}(\bm{x}) \le c_1 - c_2|\bm{x}|^2, \quad \forall \bm{x} \in \R^{dn}. \label{asmp:momentbound-dissipative}
\end{equation}
Suppose $\E|\bm{X}_0|^2 < \infty$. Then the solution of \eqref{def:mainSDE} satisfies $\sup_{t \ge 0}\E|\bm{X}_t|^2 < \infty$. Moreover, the random variables $(|\bm{X}_t|^2)_{t \ge 0}$ are uniformly integrable.
\end{lemma}
\begin{proof}
By It\^o's formula, using the representation \eqref{def:mainSDE-probabilistic}, we have
\begin{align*}
d\big(e^{2c_2t}|X^i_t|^2) &= 2\big(de^{2c_2t} +  e^{2c_2t}X^i_t \cdot \E[b^i(\bm{X}_t)\,|\,X^i_t] +  c_2 e^{2c_2t}|X^i_t|^2\big)dt + 2\sqrt{2} e^{2c_2t}X^i_t \cdot dB^i_t.
\end{align*}
Integrate to get
\begin{align}
|X^i_t|^2 &= e^{-2c_2t}|X^i_0|^2   + 2e^{-2c_2t}\int_0^t  e^{2c_2s}\big(d + X^i_s \cdot \E[b^i(\bm{X}_s)\,|\,X^i_s] +  c_2  |X^i_s|^2 \big)ds \nonumber \\
	&\qquad  + 2\sqrt{2} e^{-2c_2t}\int_0^t  e^{2c_2s}  X^i_s \cdot dB^i_s. \label{pf:momentbound1}
\end{align}
The stochastic integral is a true martingale by square-integrability of $\bm{X}$, from Proposition \ref{pr:wellposed}. 
Moveover, using \eqref{asmp:momentbound-dissipative}, we find that
\begin{align*}
\sum_{i=1}^n\E\big[X^i_s \cdot \E[b^i(\bm{X}_s)\,|\,X^i_s] +  c_2 e^{2c_2s}|X^i_s|^2\big] &= \sum_{i=1}^n \E\big[X^i_s \cdot b^i(\bm{X}_s) +  c_2  |X^i_s|^2\big] \\
	&= \E\big[ \bm{X}_s \cdot \bm{b}(\bm{X}_s) +  c_2 |\bm{X}_s|^2\big] \le c_1.
\end{align*}
 Hence, we may take expectations in \eqref{pf:momentbound1} to get
\begin{align*}
\E|\bm{X}_t|^2 &= \sum_{i=1}^n\E|X^i_t|^2 \le e^{-2c_2t}\E|\bm{X}_0|^2 + \int_0^t 2(dn+c_1)e^{-2c_2(t-s)}ds \le e^{-2c_2t}\E|\bm{X}_0|^2 + \frac{nd+c_1}{c_2}.
\end{align*}
This proves the first claim. To prove the second, let $\epsilon > 0$ and let $r > 0$, and multiply both sides of \eqref{pf:momentbound1} by $1_{\{|X^i_t|^2 > r\}}$:
\begin{align}
|X^i_t|^21_{\{|X^i_t|^2 > r\}}  &= e^{-2c_2t} |X^i_0|^21_{\{|X^i_t|^2 > r\}} \nonumber \\
	&\qquad + 2e^{-2c_2t}\int_0^t  e^{2c_2s} \Big(d + X^i_s \cdot \E[b^i(\bm{X}_s)\,|\,X^i_s] +  c_2  |X^i_s|^2\Big) 1_{\{|X^i_t|^2 > r\}}\, ds \nonumber \\
	&\qquad  + 2\sqrt{2}  1_{\{|X^i_t|^2 > r\}}e^{-2c_2t}\int_0^t  e^{2c_2s}  X^i_s \cdot dB^i_s . \label{pf:momentbound2}
\end{align}
Let $M := \sup_{t \ge 0}\E|\bm{X}_t|^2$, which we know to be finite from the first part of the proof.
Find $r' > 0$ such that $\PP(|X^i_0|^2 > r') \le \epsilon$ for all $i$.
Then the first term is handled via
\begin{align}
\E\Big[ |X^i_0|^21_{\{|X^i_t|^2 > r\}} \Big] &\le r'\PP(|X^i_t|^2 > r) + \E\big[|X^i_0|^21_{\{|X^i_0|^2 > r'\}}\big] \nonumber \\
	&\le \frac{r'}{r}M + \E\big[|X^i_0|^21_{\{|X^i_0|^2 > r'\}}\big]. \label{pf:MB-est1}
\end{align}
The second term in \eqref{pf:momentbound2} requires some care. First, note by independence that
\begin{equation*}
\PP(|X^i_t|^2 > r\,|\,X^i_s) = \PP(|X^i_t|^2 > r\,|\,\bm{X}_s) \le \PP(|\bm{X}_t|^2 > r\,|\,\bm{X}_s).
\end{equation*}
Hence, summing over $i$ and taking expectations, we have
\begin{align*}
\sum_{i=1}^n &\E\Big[\Big(d + X^i_s \cdot \E[b^i(\bm{X}_s)\,|\,X^i_s] +  c_2  |X^i_s|^2\Big) 1_{\{|X^i_t|^2 > r\}} \Big] \\
	&= \sum_{i=1}^n\E\Big[\Big(d + X^i_s \cdot  \E[b^i(\bm{X}_s)\,|\,X^i_s]  +  c_2  |X^i_s|^2\Big) \PP(|X^i_t|^2 > r\,|\,X^i_s) \Big] \\
	&= \sum_{i=1}^n\E\Big[\Big(d + X^i_s \cdot  b^i(\bm{X}_s)  +  c_2  |X^i_s|^2\Big) \PP(|X^i_t|^2 > r\,|\,X^i_s) \Big] \\
	&\le \sum_{i=1}^n\E\Big[\Big(d + X^i_s \cdot  b^i(\bm{X}_s)  +  c_2  |X^i_s|^2\Big) \PP(|\bm{X}_t|^2 > r\,|\,\bm{X}_s) \Big] \\
	&= \E\Big[\Big(nd + \bm{X}_s \cdot  b^i(\bm{X}_s)  +  c_2  |\bm{X}_s|^2\Big) \PP(|\bm{X}_t|^2 > r\,|\,\bm{X}_s) \Big] \\
	&\le (nd+c_1)\PP(|\bm{X}_t|^2 > r) \le (nd+c_1)M/r.
\end{align*}
The last step used the assumption \eqref{asmp:momentbound-dissipative} and the tower property.
The final term in \eqref{pf:momentbound2} is handled by Cauchy-Schwarz,
\begin{align*}
\E\bigg[1_{\{|X^i_t|^2 > r\}}e^{-2c_2t}\int_0^t  e^{2c_2s}  X^i_s \cdot dB^i_s  \bigg] &\le \PP(|X^i_t|^2 > r)^{1/2}\E\bigg[\int_0^t e^{-4c_2(t-s)}|X^i_s|^2\,ds\bigg]^{1/2} \!\le \! \frac{M}{ \sqrt{4rc_2}}.
\end{align*}
We now take expectations in \eqref{pf:momentbound2},  summing over $i=1,\ldots,n$ and combining the above estimates, to find
\begin{align*}
\sum_{i=1}^n\E\big[|X^i_t|^21_{\{|X^i_t|^2 > r\}}\big] &\le \frac{nr'M}{r} + \sum_{i=1}^n\E\big[|X^i_0|^21_{\{|X^i_0|^2 > r'\}}\big] + \frac{(d+c_1)M}{c_2r} + \frac{M}{ \sqrt{4rc_2}}.
\end{align*}
Because $\E|X^i_0|^2 < \infty$ for each $i$, we may send $r\to\infty$ followed by $r' \to \infty$ to deduce that
\begin{equation*}
\lim_{r\to\infty}\sup_{t \ge 0} \E\big[|X^i_t|^21_{\{|X^i_t|^2>r\}}\big] = 0. \qedhere
\end{equation*}
\end{proof}

\subsection{An entropy identity}

We state here a well known identity for the relative entropy between the laws of two diffusion processes. In the case of bounded drifts, this follows easily from Girsanov's theorem, while the general case requires a somewhat more careful approximation. We state a version good enough to cover our needs:

\begin{lemma}[Lemma 4.4(i) of \cite{lacker2023hierarchies}] \label{le:entropy-pathspace}
Let $k \in \N$ and $T > 0$, and let $b^1,b^2 : [0,T] \times \R^k \to \R^k$ be Borel measurable. Suppose $x \mapsto b^2(t,x)$ is Lipschitz, uniformly in $t$. For each $i=1,2$, let $Z^i$ be a weak solution of the SDE
\begin{align*}
dZ^i_t = b^i(t,Z^i_t)dt + \sqrt{2}\,dB_t.
\end{align*}
Let  $P^i \in \P(C([0,T];\R^k))$ denote the law of $(Z^i_t)_{t \in [0,T]}$.  Assume also that
\begin{align}
\E\int_0^T|b^1(t,Z^i_t)-b^2(t,Z^i_t)|^2\,dt < \infty, \quad i=1,2. \label{asmp:entropy-girsanov}
\end{align}
Then
\begin{align*}
H(P^1\,|\,P^2) = H(P^1_0\,|\,P^2_0) + \frac{1}{4}\E\int_0^T |b^1(t,Z^1_t)-b^2(t,Z^1_t)|^2\,dt.
\end{align*}
\end{lemma}

\subsection{Remarks on symmetry}

Here we record an interesting general principle: the independent projection strengthens symmetries.

As a first example, suppose the drift $\bm{b}$ commutes with permutations in the sense that $\bm{b}(s\bm{x})=s\bm{b}(\bm{x})$ for all permutations $s$ and all $\bm{x} \in (\R^d)^n$. Here, a permutation $s$ of $[n]=\{1,\ldots,n\}$ acts on a vector $\bm{x} =(x^1,\ldots,x^n)$ via $s\bm{x} = (x^{s(1)},\ldots,x^{s(n)})$. Then, the original SDE \eqref{def:SDE-reference} is obviously exchangeable in the sense that $s\bm{Y}\stackrel{d}{=}\bm{Y}$ for all permutations $s$, if it is initialized as such at time zero. This transfers to the independent projection as well, so that $\bm{X}=(X^1,\ldots,X^n)$ is also exchangeable. But an exchangeable vector of independent random variables is in fact iid. In short, exchangeability of the original SDE leads to iid independent projections. 

This first example generalizes considerably. Consider any group $G$ of permutations of $[n]$, and suppose that $\bm{b}$ commutes with this group in the sense that $\bm{b}(s\bm{x})=s\bm{b}(\bm{x})$ holds for $s \in G$ and $\bm{x} \in (\R^d)^n$. Then $s\bm{X} \stackrel{d}{=} \bm{X}$ for all $s \in G$, which by independence is equivalent to $X^{s(i)} \stackrel{d}{=} X^i$ for all $s\in G$ and all $i \in [n]$. In other words, $X^i \stackrel{d}{=} X^j$ whenever there exists $s \in G$ such that $s(i)=j$; that is, the map $i \mapsto \mathrm{Law}(X^i)$ is constant along orbits of the action of  $G$ on $[n]$. In particular, if $G$ acts \emph{transitively} on $[n]$, then $(X^1,\ldots,X^n)$ are again iid.

\section{Proof of the gradient flow Theorem \ref{th:JKO}} \label{se:proof:JKO}

This section gives the proof of Theorem \ref{th:JKO}, which is in fact just a minor modification of the proof given in the original Jordan-Kinderlehrer-Otto paper \cite{jordan1998variational}. As such, we will omit many details, so that we may focus on the few new ideas. See also \cite[Section 8.3]{santambrogio2015optimal} for an excellent alternative  exposition of the JKO proof.

Note that we may equivalently write the optimization problem \eqref{def:JKO} as\footnote{It is well known that a probability density $\mu$ with $\int|x|^2\mu(x)dx < \infty$ satisfies $\int (\mu\log\mu)_- <\infty$, so that $H(\mu)$ is well-defined in $(-\infty,\infty]$; see the remark on page 9 of \cite{jordan1998variational}.}
\begin{equation*}
\inf_{\mu \in \P_2^{\otimes n}(\R^d)} \bigg( H(\mu) - \int_{\R^{dn}} f\,d\mu + \frac{1}{2\tau}\W_2^2(\mu,\mu_{\tau,k})\bigg).
\end{equation*}
The key point in our context is to make use of two straightforward and well known tensorization properties, valid for any $\mu,\nu \in \P^{\otimes n}_2(\R^d)$:
\begin{align*}
H(\mu) =\sum_{i=1}^n H(\mu^i), \qquad \W_2^2(\mu,\nu) = \sum_{i=1}^n \W_2^2(\mu^i,\nu^i).
\end{align*}
With these identities in mind, if $\mu_{\tau,k+1}=\mu^1_{\tau,k+1} \otimes \cdots \otimes \mu^n_{\tau,k+1}$ is an optimizer for \eqref{def:JKO}, then for each coordinate $i$ we see that $\mu^i_{\tau,k+1}$ must be an optimizer of
\begin{equation}
\inf_{\nu \in \P_2(\R^d)}\bigg( H(\nu) - \int_{\R^d} f_{\mu_{\tau,k+1}}^i\,d\nu + \frac{1}{2\tau}\W_2^2(\nu,\mu^i_{\tau,k})\bigg). \label{eq:JKOcoord-i}
\end{equation}
Here, for $\mu \in \P^{\otimes n}_2(\R^d)$ and $i \in \{1,\ldots,n\}$, we define the conditional expectation $f_\mu^i : \R^d \to \R$ by
\begin{equation*}
f_\mu^i(x^i) := \int_{\R^{d(n-1)}} f(x^i,\bm{x}^{-i})\,\mu^{-i}(d\bm{x}^{-i}).
\end{equation*}
This integral is well-defined because $\mu$ has finite second moments and $f$ has quadratic growth, because $f$ has bounded second order derivatives. In addition, $\nabla f_\mu^i$ is Lipschitz with the same constant as $\nabla f$.

Now, the variational problem \eqref{eq:JKOcoord-i} satisfied by coordinate $i$ is exactly the same as the one studied in the original JKO paper \cite{jordan1998variational}, with the function $\Psi$ therein taken to be $\Psi=-f_{\mu_{\tau,k+1}}^i$.
We may thus follow verbatim the beginning of the proof of \cite[Theorem 5.1]{jordan1998variational}, all the way to the inequality (41) therein:
\begin{equation}
\bigg|\int_{\R^d}\Big(\frac{1}{\tau}(\mu^i_{\tau,k}-\mu^i_{\tau,k-1}) \varphi - (\nabla\varphi \cdot \nabla f_{\mu_{\tau,k}}^i+ \Delta \varphi)\mu^i_{\tau,k}\Big)\bigg| \le \frac{\|\nabla^2\varphi\|_\infty}{2\tau}\W_2^2(\mu^i_{\tau,k},\mu^i_{\tau,k-1}), \label{pf:JKOidentity1}
\end{equation}
where the integral is with respect to Lebesgue measure. The same arguments leading to (45) of \cite{jordan1998variational} show that, for each $T > 0$,
\begin{align}
 \sum_{k=1}^{\lceil T/\tau\rceil}\W_2^2(\mu^i_{\tau,k},\mu^i_{\tau,k-1}) &\le C\tau, \label{pf:JKO-W2bound} \\
 \int_{\R^d}|x^i|^2\mu^i_{\tau,k}(dx^i) &\le C, \ \forall k=0,\ldots,\lceil T/\tau\rceil, \label{pf:JKO-momentbound}
\end{align}
for some constant $C$ which may depend on $T$ but not on $i$, $k$, or $\tau$.
As argued in \cite[(47)]{jordan1998variational}, this allows us to extract a subsequence such that $\mu_{\tau} \to \mu$ weakly in $L^1((0,T) \times \R^{dn})$ as $\tau \to 0$ for each $T > 0$, and also $\mu_{\tau}(t) \to \mu(t)$ weakly in $L^1(\R^{dn})$ for each $t > 0$. 
Note also that $\nabla f$ has linear growth, being Lipschitz, so that \eqref{pf:JKO-momentbound} implies
\begin{align}
\int_{\R^d} |\nabla f_{\mu_{\tau,k}}^i(x^i)| \mu_{\tau,k}^i(dx^i) &\le \int_{\R^{dn}} |\nabla_i f(\bm{x})| \mu_{\tau,k}(d\bm{x}) \le C, \ \forall k=0,\ldots,\lceil T/\tau\rceil, \label{pf:JKO-momentbound2}
\end{align}
for some other constant $C$.

With these preparations, we now explain how to send $\tau \downarrow 0$ in \eqref{pf:JKOidentity1}.
Let $\varphi \in C^\infty_c((-\infty,T) \times \R^{d})$, and set $m_\tau =\lceil T/\tau\rceil$. Then, by summation by parts
\begin{align*}
\int_{(0,T) \times \R^d} \mu^i_{\tau}(t,x) \partial_t\varphi(t,x)\,dxdt &= \sum_{k=1}^{ m_\tau+1 }\int_{(k-1)\tau}^{k\tau} \int_{\R^d} \mu^i_{\tau}(t,x) \partial_t\varphi(t,x)\,dxdt \\
	&= \sum_{k=1}^{ m_\tau+1}  \int_{\R^d}  \mu^i_{\tau,k-1}(x) (\varphi(k\tau,x)- \varphi((k-1)\tau,x))\,dxdt \\
	&= -\sum_{k=1}^{ m_\tau }  \int_{\R^d}  (\mu^i_{\tau,k}(x)-\mu^i_{\tau,k-1}(x))  \varphi(k\tau,x) \,dxdt \\
	&\qquad - \int_{\R^d}   \varphi(0,x)\,\mu^i_{0}(x)dx.
\end{align*}
On the other hand, using the fact that $\nabla f$ has quadratic growth along with  \eqref{pf:JKO-momentbound}, we get
\begin{align*}
\int_{(0,T) \times \R^d} &\mu^i_{\tau}(t,x)(\nabla \varphi(t,x)\cdot \nabla f_{\mu_{\tau}(t)}^i(x) + \Delta \varphi(t,x))  \,dxdt \\
	&= \sum_{k=0}^{ m_\tau }\int_{k\tau}^{(k+1)\tau} \int_{\R^d}\mu^i_{\tau,k}(x)(\nabla \varphi(t,x)\cdot \nabla f_{\mu_{\tau,k}}^i(x) + \Delta \varphi(t,x))\,dxdt \\
	&= (m_\tau+1)O(\tau^2) + \sum_{k=0}^{ m_\tau } \tau\int_{\R^d}\mu^i_{\tau,k}(x)(\nabla \varphi(k\tau,x)\cdot \nabla f_{\mu_{\tau,k}}^i(x) + \Delta \varphi(k\tau,x))\,dx ,
\end{align*}
where the constant hidden in  $O(\tau^2)$ depends only on  $\varphi$ and  the constant of \eqref{pf:JKO-momentbound2}.
Note that $m_\tau O(\tau^2) = O(\tau)$ as $\tau \downarrow 0$.
Adding these two identities and recalling \eqref{pf:JKOidentity1} and \eqref{pf:JKO-W2bound}, we find
\begin{align*}
\int_{(0,T) \times \R^d} &\mu^i_{\tau}(t,x)(\partial_t\varphi(t,x) + \nabla \varphi(t,x)\cdot \nabla f_{\mu_{\tau}(t)}^i(x) + \Delta \varphi(t,x))  \,dxdt = O(\tau) + \int_{\R^d}  \varphi(0,x)\,\mu^i_{0}(x)dx.
\end{align*}
Note also by definition of $f_{\mu_{\tau}(t)}^i$ that
\begin{align*}
\int_{\R^d} \nabla \varphi(t,x)\cdot \nabla f_{\mu_{\tau}(t)}^i(x)\mu^i_{\tau}(t,x)\,dx &= \int_{\R^{dn}} \nabla \varphi(t,x^i)\cdot \nabla_i f(\bm{x})\mu_{\tau}(t,\bm{x})\,d\bm{x}.
\end{align*}
Combining the two previous equations, we may use the (subsequential) convergence $\mu_{\tau} \to \mu$ to deduce that
\begin{align*}
\int_{(0,T) \times \R^d} &\mu^i(t,x)(\partial_t\varphi(t,x) + \nabla \varphi(t,x)\cdot \nabla f_{\mu(t)}^i(x) + \Delta \varphi(t,x))  \,dxdt = \int_{\R^d}  \varphi(0,x)\,\mu^i_{0}(x)dx,
\end{align*}
which hold for each $i$ and each $\varphi \in C^\infty_c((-\infty,T) \times \R^{d})$.
This shows that $(\mu^i(t))_{t \ge 0}$ is a weak (measure) solution of the Fokker-Planck equation; if we show that it is weakly continuous, we can apply the superposition principle \cite[Theorem 2.5]{trevisan2016well} to deduce that $\mu^i(t)$ must be the time-marginals of a solution of the corresponding SDE \eqref{def:mainSDE} with $b^i=\nabla_if$. By uniqueness of the SDE, this completes the proof, pending the claimed continuity.

To prove weak continuity of $t \mapsto \mu(t)$, we borrow an argument from \cite{santambrogio2015optimal}, which circumvents PDE arguments used to finish the proof in  \cite{jordan1998variational}.
Let $\widetilde\mu_\tau^i(t)$ for $t \in [(k-1)\tau,k\tau]$ be a $\W_2$-geodesic connecting $\mu^i_{\tau,k-1}$ to $\mu^i_{\tau,k}$. Then $\widetilde\mu^i_\tau(t)$ clearly shares with $\mu^i_\tau(t)$ the same weak limit $\mu^i$  (in the probabilist's sense, and along the same subsequence). The same estimates leading to \cite[inequality (8.1.2)]{santambrogio2015optimal} yield
\begin{align*}
\W_2(\widetilde\mu^i_{\tau}(t),\widetilde\mu^i_{\tau}(s)) &\le C \sqrt{(t-s)},
\end{align*}
for some constant $C$.
Note also that \eqref{pf:JKO-momentbound} implies that $\{\widetilde\mu_\tau(t) : t \le T, \tau \in (0,1)\}$ is tight. We can then apply the Arzel\`a-Ascoli theorem to deduce that the weak convergence $\widetilde\mu_\tau(t) \to \mu(t)$ happens uniformly in $t \in [0,T]$. Thus $\mu$ must itself be weakly continuous, and the proof is complete. \hfill \qedsymbol

\section{Proof of long-time behavior} \label{se:proof:longtime}

This section proves the claims of Section \ref{se:longtime}, begining with some preliminary lemmas. Throughout this section, $f$ is assumed as usual to have Lipschitz gradient, and also here it satisfies \eqref{asmp:dissipative}. And, as usual, $\bm{X}$ is the solution of \eqref{def:mainSDE}, with $\mu[T] \in \P(C([0,T];\R^{dn}))$ denoting the law of $(\bm{X}_t)_{t \in [0,T]}$, and $\mu_t$ denoting the law of $\bm{X}_t$. 

\subsection{Preliminaries}

As a first preparation, we show that entropy becomes finite instantaneously, even if it is not initially so. 

\begin{lemma} \label{le:entropycomesdown}
For each $t > 0$, we have $H(\mu_t\,|\,\rho_*) < \infty$.
\end{lemma}
\begin{proof}
Let $P[t]$ denote the law of $(\bm{X}_0+ \bm{B}_s)_{s \in [0,t]}$, a Brownian motion initialized from law $\mu_0$. Then, using Lemma \ref{le:entropy-pathspace} (with $b^2\equiv 0$),
\begin{align*}
H(\mu[t]\,|\,P[t]) &= \frac14\E\sum_{i=1}^n\int_0^t|\E[\nabla_if(\bm{X}_s)\,|\,X^i_s]|^2\,ds \le \frac14\E \int_0^t|\nabla f(\bm{X}_s) |^2\,ds < \infty,
\end{align*}
because $\nabla f$ has linear growth and $\bm{X}_s$ is integrable by Proposition \ref{pr:wellposed}. By the data processing inequality for relative entropy, $H(\mu_t \,|\, \mu_0 * \gamma_t) \le H(\mu[t] \,|\, P[t]) < \infty$, where $\gamma_t$ is the centered Gaussian measure in $\R^{dn}$ with covariance matrix $tI$. Then, note that
\begin{align*}
H(\mu_t\,|\,\rho_*) &= H(\mu_t\,|\,\mu_0 * \gamma_T) - \int f\,d\mu_t + \int \log (\mu_0* \gamma_t)\,d\mu_t.
\end{align*}
We just saw that $H(\mu_t\,|\,\mu_0 * \gamma_T) < \infty$. We also have $f \in L^1(\mu_t)$ because $f$ has quadratic growth and $\mu_t$ has finite second moment. Finally, note that $\mu_0* \gamma_t \le (2\pi t)^{-dn/2}$ pointwise.
\end{proof}

Next, we show that the measure flow $(\mu_t)$ is uniformly continuous:

\begin{lemma} \label{le:unifcont}
The map $\R_+ \ni t \mapsto \mu_t \in (\P_2(\R^{dn}),\W_2)$ is uniformly continuous.
\end{lemma}
\begin{proof}
Recall that $\nabla f$ is Lipschitz, say with constant $L$.
Noting that $(\bm{X}_t,\bm{X}_s)$ couples $(\mu_t,\mu_s)$, 
\begin{align*}
\W_2^2(\mu_t,\mu_s) &\le \E|\bm{X}_t-\bm{X}_s|^2 \le 2(t-s)\sum_{i=1}^n\int_s^t\E |\E[\nabla_if(\bm{X}_r)\,|\,X^i_r]|^2\,dr + 4\E|\bm{B}_t-\bm{B}_s|^2 \\
	&\le 4(t-s)^2|\nabla f(0)|^2 + 4L^2(t-s)^2\sup_{r \ge 0}\E|\bm{X}_r|^2 + 4dn(t-s).
\end{align*}
The assumption \eqref{asmp:dissipative} lets us apply Lemma \ref{le:momentbound} to deduce that the supremum is finite.
\end{proof}

\begin{remark}
While we opted for direct proofs using the SDE, it is worth noting that Lemmas \ref{le:entropycomesdown} and \ref{le:unifcont} could also be deduced quickly from the convergence of the JKO scheme in Theorem \ref{th:JKO} and estimates thereof.
\end{remark}

We lastly record the stratightforward observation that the solutions of the mean field equations are invariant measures for the SDE. The converse is true and even more straightforward to prove, but we will not need it.

\begin{lemma} \label{le:MF-invariant}
Suppose $\mu_0$ solves the mean field equations \eqref{def:MFequations}. Then $\mu_t = \mu_0$ for all $t \ge 0$.
\end{lemma}
\begin{proof}
Differentiate the mean field equations to find, because $\mu_0$ is a product measure, that
\begin{equation*}
\nabla \log \mu^i_0(x^i) = \nabla\E_{\mu_0}[  f(\bm{X})\,|\,X^i=x^i] = \E_{\mu_*}[\nabla_i f(\bm{X})\,|\,X^i=x^i].
\end{equation*}
Letting $\nu_t=\mu_0$ for all $t \ge 0$, integration by parts shows that
\begin{equation*}
 \frac{d}{dt}\int \varphi\,d\nu^i_t  = 0 = \E_{\mu_0}\big[ \nabla\varphi(X^i) \cdot \E[\nabla_i f(\bm{X})\,|\,X^i ] + \Delta\varphi(X^i )\big].
\end{equation*}
By the superposition principle \cite[Theorem 2.5]{trevisan2016well}, there exists a solution $\bm{X}'$ of the SDE \eqref{def:mainSDE} such that $\bm{X}'_t \sim \nu_t$ for all $t$. By $\nu_0=\mu_0$ and the uniqueness of the SDE (Proposition \ref{pr:wellposed}), we have $\bm{X}_t \stackrel{d}{=} \bm{X}'_t$ for all $t \ge 0$, or $\mu_t=\nu_t = \mu_0$.
\end{proof}

As a final preparation, we prove a crucial stability property for $\widetilde{I}$, which was defined in \eqref{def:projectedFisher}.

\begin{lemma} \label{le:coercive}
Suppose $\nu_m \in \P^{\otimes n}_2(\R^d)$ satisfies 
\begin{equation}
\lim_{m\to\infty} \widetilde{I}(\nu_m\,|\,\rho_*) = 0. \label{asmp:keyI}
\end{equation}
If $\W_2(\nu_m,\nu) \to 0$ for some $\nu \in \P^{\otimes n}_2(\R^d)$, then $\nu$ satisfies  the mean field equations \eqref{def:MFequations}.
\end{lemma}
\begin{proof}
Let us assume without loss of generality that $\nu_m$ admits a strictly positive density for each $m$; if it does not, we may simply replace it with $\bar\nu_m := (1-(1/m))\nu_m +(1/m)\gamma$, where $\gamma$ is the standard Gaussian, and note easily that $\bar\nu_m$ satisfies the same assumption \eqref{asmp:keyI} as well as $\W_2(\bar\nu_m,\nu) \to 0$. 

Fix $i\in\{1,\ldots,n\}$, and define 
\begin{align*}
\xi_m^i(x^i) := \E_{\nu_m}[\nabla_i f(\bm{X})\,|\,X^i=x^i] - \nabla \log \nu_m^i(x^i).
\end{align*}
We may assume $\widetilde{I}(\nu_m\,|\,\rho_*) < \infty$ for all $m$, which implies that $\log\nu_m^i$ is weakly differentiable with $\nu_m^i$-square-integrable gradient.
The assumption \eqref{asmp:keyI} states that $\E_{\nu_m}|\xi_m^i(X^i)|^2  \to 0$. For any $h \in C^\infty_c(\R^d;\R^d)$, we have
\begin{align*}
\E_{\nu_m}[\xi_m^i(X^i) \cdot h(X^i)] &= \E_{\nu_m}\big[ \nabla_i f(\bm{X}) \cdot h(X^i) - \nabla_i\log \nu_m^i(X^i) \cdot h(X^i)\big].
\end{align*}
By integration by parts, the last term rewrites as
\begin{align*}
\E_{\nu_m}[\nabla \log \nu_m^i(X^i) \cdot h(X^i)\big]=\int_{\R^d} h(x^i) \cdot \nabla\nu_m^i(x^i) \,dx^i = -\E_{\nu_m}[\mathrm{div}\,h(X^i)],
\end{align*}
where we used strict positivity of $\nu_m$ to avoid boundary terms.
Hence,
\begin{align*}
0 &= \lim_m\E_{\nu_m}[\xi_m^i(X^i) \cdot h(X^i)] = \lim_m\E_{\nu_m}\big[ \nabla_i f(\bm{X}) \cdot h(X^i) + \mathrm{div}\, h(X^i)\big] \\
	&= \E_{\nu}\big[ \nabla_i f(\bm{X}) \cdot h(X^i) + \mathrm{div}\, h(X^i)\big],
\end{align*}
where the last limit follows from the weak convergence $\nu_m\to\nu$ and boundedness of $h \cdot \nabla f$ and $\mathrm{div} \,h$. 
Rearranging,
\begin{align*}
\int_{\R^d} \mathrm{div}\,h(x^i) \, \nu^i_m(dx^i) &= - \int_{\R^{dn}}\nabla_i f(\bm{x}) \cdot h(x^i)\,\nu_m(d\bm{x}) \\
	&=  - \int_{\R^d}\E_\nu[\nabla_i f(\bm{X})\,|\,X^i=x^i] \cdot h(x^i)\,\nu_m^i(dx^i).
\end{align*}
As $h$ was arbitrary, we deduce the identity of weak derivatives:
\begin{equation*}
\nabla_i \nu^i(x^i) = \nu^i(x^i)\E[\nabla_i f(\bm{X})\,|\,X^i=x^i]. \qedhere
\end{equation*}
\end{proof}

\subsection{Proof of Theorem \ref{th:main-MF}} \label{se:longtime:proofs}

The convergence analysis revolves around a computation of relative entropy. 
By Lemma \ref{le:entropycomesdown}, we may shift time and assume that $H(\mu_0\,|\,\rho_*) <\infty$.
Let $b^i_t(x^i)=\int \nabla_i f(x^i,\bm{x}^{-i})\,\mu_t^{-i}(d\bm{x}^{-i})$ denote the drift of $\mu^i$, and $\bm{b}_t(\bm{x})=(b^1_t(x^1),\ldots,b^n_t(x^n))$. We perform a well known formal calculation using the Fokker-Planck equation \eqref{def:FokkerPlanck}: 
\begin{align*}
\frac{d}{dt}H(\mu_t\,|\,\rho_*) &= \frac{d}{dt}\int \mu_t\log\frac{\mu_t}{\rho_*} = \int  \log\frac{\mu_t}{\rho_*}\partial_t\mu_t \\
	&= \int  \log\frac{\mu_t}{\rho_*}\big(-\mathrm{div}(\mu_t\bm{b}_t) + \Delta\mu_t\big) \\
	&= \int  \nabla \log\frac{\mu_t}{\rho_*} \cdot \big(\bm{b}_t - \nabla\log\mu_t\big) \,\mu_t,
\end{align*}
where the integrals are with respect to Lebesgue measure on $\R^{dn}$. 
Using $\nabla\log\rho_*=f$ and $b^i_t(x^i)=\E[\nabla_if(\bm{X}_t)\,|\,X^i_t=x^i]$,  this rewrites as
\begin{align}
\frac{d}{dt}H(\mu_t\,|\,\rho_*) &= \sum_{i=1}^n\E\Big[\big(\nabla\log\mu^i_t(X^i_t) - \nabla_i f(\bm{X}_t)\big) \cdot \big(\E[\nabla_if(\bm{X}_t)\,|\,X^i_t]-\nabla\log\mu^i_t(X^i_t)\big)\Big] \nonumber \\
	&= -\sum_{i=1}^n\E\big|\E[\nabla_i f(\bm{X}_t)\,|\,X^i_t] - \nabla_i\log \mu^i_t(X^i_t)\big|^2 \nonumber \\
	&= -\widetilde{I}(\mu_t\,|\,\rho_*). \label{pf:entropycalculation}
\end{align}
where $\widetilde{I}$ was defined in \eqref{def:projectedFisher}.
This implies, for $T > 0$, that
\begin{equation}
H(\mu_T\,|\,\rho_*) + \int_0^T\widetilde{I}(\mu_t\,|\,\rho_*)\,dt \le H(\mu_0\,|\,\rho_*). \label{ineq:entropycalculation}
\end{equation}
Because $H \ge 0$ and $H(\mu_0\,|\,\rho_*) < \infty$, we deduce that and
\begin{align}
\sup_{t \ge 0} H(\mu_t\,|\,\rho_*) < \infty, \qquad 
\int_0^\infty \widetilde{I}(\mu_t\,|\,\rho_*)\,dt < \infty, \label{pf:int0-inf}
\end{align}

The calculations above were formal, in the sense that a proper proof should justify the smoothness and integrability of the quantities involved. The densities $\mu_t(x)$ are all smooth on $(t,x) \in (0,\infty) \times \R^d$ because $\bm{b}$ is smooth and Lipschitz, and one needs only take limits after suitable truncations of the function $x\log x$. In fact, we only need \eqref{ineq:entropycalculation} as an inequality, not equality, which significantly simplifies these approximation arguments.  This is well-traveled terrain, so we omit details; see \cite[Proof of Lemma 2.4]{BogRocSha} for arguments of this nature.

It follows from \eqref{pf:int0-inf} that $\{\mu_t : t \ge 0\}$ is tight, because it is contained in a sublevel set of $H(\cdot\,|\,\rho_*)$. The uniform integrability of Lemma \ref{le:momentbound} then implies that $\{\mu_t : t \ge 0\}$ is precompact in $(\P_2(\R^{dn}),\W_2)$.

Let $\epsilon > 0$. By uniform continuity (Lemma \ref{le:unifcont}), we may find $\delta > 0$ such that $\W_2(\mu_t,\mu_s) \le \epsilon$ for all $t,s \ge 0$ with  $|t-s| \le \delta$. By \eqref{pf:int0-inf} and Tonelli's theorem, we have
\begin{align*}
\int_0^\delta \sum_{k=0}^\infty \widetilde{I}(\mu_{u+k\delta}\,|\,\rho_*)\,du < \infty,
\end{align*}
which implies that there exists a Borel set $I_\delta \subset [0,\delta]$ of full Lebesgue measure such that
\begin{align*}
\lim_{k\to\infty}\widetilde{I}(\mu_{u+k\delta}\,|\,\rho_*) = 0,\quad \forall u \in I_\delta.
\end{align*}
By Lemma \ref{le:coercive}, for each $u \in I_\delta$ we deduce that every $\W_2$-limit point of the sequence $(\mu_{u+k\delta})_{k\in\N}$ must belong to the set $S_{\mathrm{MF}}$ of solutions of the mean field equations \eqref{def:MFequations}.
Because this sequence is precompact, it is easy to deduce that in fact  $\lim_k \W_2(\mu_{u+k\delta},S_{\mathrm{MF}})=0$ for all $u \in I_\delta$.
Now fix $u_* \in I_\delta$, and define $\tau(t)$ for each $t \ge 0$ to be the smallest element of $\{u_*+(k-1)\delta :  k \in \N\}$ which is larger than $t$. Then $|t -\tau(t)| \le \delta$, and thus $\W_2(\mu_{t},\mu_{\tau(t)}) \le \epsilon$. Using the triangle inequality,
\begin{align*}
\W_2(\mu_{t},S_{\mathrm{MF}}) \le \epsilon + \W_2(\mu_{\tau(t)},S_{\mathrm{MF}}).
\end{align*}
We know also that
\begin{align*}
\lim_{t\to\infty}\W_2(\mu_{\tau(t)},S_{\mathrm{MF}}) = \lim_{k\to\infty}\W_2(\mu_{u_*+k\delta},S_{\mathrm{MF}}) =0.
\end{align*}
As $\epsilon$ was arbitrary, we deduce that $\W_2(\mu_{t},S_{\mathrm{MF}})\to 0$.  This proves the first claim of Theorem \ref{th:main-MF}.

It remains to show the converegence of entropy \eqref{eq:entropyconvergence}. For this we appeal to regularity results from the theory of parabolic PDEs, namely a combination of an upper bound on the density with interior H\"older estimates. First, note that on any compact set $S \subset \R^d$, the function $\E[\nabla_i f(\bm{X}_t)\,|\,X^i_t=x^i]$ is bounded in $(t,x^i) \in (0,\infty) \times S$. By It\^o's formula, the measure flow $(\mu^i_t)_{t \ge 0}$ solves (in the distributional sense) the Fokker-Planck equation with this drift. Applying \cite[Corollary 7.2.2]{BKRSbook} shows that (the density of) $\mu^i_t(x^i)$ is bounded on $[\delta,\infty) \times \R^{d}$ for any $\delta > 0$. We refer to \cite[Corollary 6.4.3]{BKRSbook} for a statement of interior H\"older estimates which is suitable for our needs (and in fact much more general):
It tells us that $\mu^i_t$ admits a continuous density for each $t > 0$, and for each $R > 0$ there exist constants $(\alpha,\beta,K)$  such that
\begin{align*}
|\mu_t^i(x)-\mu_{t'}(x')| \le K(|x-x'|^\alpha + |t-t'|^{\beta}), 
\end{align*}
for each $x,x' \in \R^{d}$ with norm at most $R$, and each $t,t' \ge1$ with $|t-t'| \le 1$.
By Arzel\`a-Ascoli, for any compact set $S \subset \R^d$ the family $\{\mu^i_t|_S : t \ge 1\} \subset C(S)$ is precompact.

With these preparations, we are ready to show that $H(\mu_{t_n}\,|\,\rho_*) \to H(\mu_*\,|\,\rho_*)$, for any given $\mu_* \in S_{\mathrm{MF}}$ and any sequence $t_n \to \infty$ for which $\W_2(\mu_{t_n}, \mu_*) \to 0$.  By the above application of Arzel\`a-Ascoli, we know that in fact the density $\mu_{t_n}$ converges pointwise to that of $\mu_*$, uniformly on compact sets in $\R^{dn}$.
For $r > 0$, let $B_r$ denote the centered ball of radius $r$ in $\R^{dn}$. Recall that $\mu_{t_n}$ are uniformly bounded, so that $\log \mu_{t_n} \le C$ for some constant $C > 0$. Then
\begin{align*}
H(\mu_{t_n}\,|\,\rho_*) = \int_{B_r}\mu_{t_n}\log\frac{\mu_{t_n}}{\rho_*} + \int_{B_r^c}\mu_{t_n}\log\frac{\mu_{t_n}}{\rho_*}
\end{align*}
By uniform convergence and the fact that $\rho_*$ is bounded away from zero on $B_r$, we have
\begin{align*}
\int_{B_r}\mu_{t_n}\log\frac{\mu_{t_n}}{\rho_*} \to \int_{B_r}\mu_*\log\frac{\mu_*}{\rho_*}.
\end{align*}
For the other term, use $\log(\mu_{t_n}/\rho_*) \le C - f$ to deduce
\begin{align*}
\int_{B_r^c}\mu_{t_n}\log\frac{\mu_{t_n}}{\rho_*} \le  \int_{B_r^c}(C+|f|)\,\mu_{t_n}.
\end{align*}
Because $f$ has quadratic growth, the uniform integrability of Lemma \ref{le:momentbound} ensures that
\begin{align*}
\lim_{r\to\infty}\sup_n\int_{B_r^c}(C+|f|)\,\mu_{t_n} = 0.
\end{align*}
Using these observations, we may send $n\to\infty$ and then $r\to\infty$ to deduce that $H(\mu_{t_n}\,|\,\rho_*) \to H(\mu_*\,|\,\rho_*)$, completing the proof. \hfill \qedsymbol

\subsection{Proof of Theorem \ref{th:convex-limit}}

The claims of Theorem \ref{th:convex-limit}(i,ii) follow immediately from Theorem \ref{th:main-MF}.
Hence, we prove only (iii) here. 
We couple two solutions of \eqref{def:mainSDE} as follows.
On some probability space, consider random vectors $(X^i_0,\overline{X}^i_0)$ distributed according to an optimal coupling for $\W_2^2(\mu_0^i,\mu_*^i)$, and take them to be independent across $i=1,\ldots,n$. Let $\bm{X}_0=(X^1_0,\ldots,X^n_0)$ and $\bm{\overline{X}}=(\overline{X}^1_0,\ldots,\overline{X}^n_0)$. Because $\mu_0$ and $\mu_*$ are product measures, it is straightforward to check that $(\bm{X}_0,\bm{\overline{X}}_0)$ is an optimal coupling for $\W_2^2(\mu_0,\mu_*)$.
Let $\bm{B}$ be a Brownian motion, independent of $(\bm{X}_0,\bm{\overline{X}}_0)$.
Let $(\F^i_t)_{t \ge 0}$ be the filtration generated by the process $(X^i_0,\overline{X}^i_0,B^i_t)_{t \ge 0}$.
In this probability space, let $\bm{X}$ and $\bm{\overline{X}}$ respectively denote the unique strong solutions of the SDE \eqref{def:mainSDE}, driven by the same Brownian motion $\bm{B}$, and initialized respectively from $\bm{X}_0$ and $\bm{\overline{X}}_0$. Note by Lemma \ref{le:MF-invariant} that $\bm{\overline{X}}_t \sim \mu_*$ for all $t > 0$.

Note that $\F^i_t$ and $\F^j_t$ are independent for $i \neq j$, by construction.
By strong well-posedness, $X^i$ is adapted to $(\F^i_t)$. By independence, 
\begin{equation*}
\E[\nabla_i f(\bm{X}_t)\,|\,X^i_t] = \int_{\R^{d(n-1)}} \nabla_if(X^i_t,\bm{x}^{-i})\,\mu_t^{-i}(d\bm{x}^{-i})  
= \E[\nabla_i f(\bm{X}_t)\,|\,\F^i_t].
\end{equation*}
Similarly with $\bm{\overline{X}}$ in place of $\bm{X}$. Then, by It\^o's formula,
\begin{align*}
d|X^i_t-\overline{X}^i_t|^2 &= 2(X^i_t-\overline{X}^i_t) \cdot \big(\E[\nabla_i f(\bm{X}_t)\,|\,X^i_t] - \E[\nabla_i f(\bm{\overline{X}}_t)\,|\,\overline{X}^i_t]\big)dt \\
	&= 2\E\Big[(X^i_t-\overline{X}^i_t) \cdot (\nabla_i f(\bm{X}_t) - \nabla_i f(\bm{\overline{X}}_t)) \,\Big|\,\F^i_t\Big].
\end{align*}
Integrate and use the tower property to get
\begin{align*}
\frac{d}{dt}\E|\bm{X}_t- \bm{\overline{X}}_t|^2 &= 2\sum_{i=1}^n \E\Big[(X^i_t-\overline{X}^i_t) \cdot (\nabla_i f(\bm{X}_t) - \nabla_i f(\bm{\overline{X}}_t)) \Big] \\
	&= 2\E\big[ (\bm{X}_t - \bm{\overline{X}}_t) \cdot (\nabla f(\bm{X}_t) - \nabla f(\bm{\overline{X}}_t))\big] \\
	&\le - 2 \kappa \E|\bm{X}_t- \bm{\overline{X}}_t|^2. 
\end{align*}
By Gr\"onwall's inequality and the fact that $(\bm{X}_0,\bm{\overline{X}}_0)$ is optimally coupled,
\begin{equation*}
\W_2^2(\mu_t,\mu_*) \le \E|\bm{X}_t- \bm{\overline{X}}_t|^2 \le e^{-2\kappa t}\E|\bm{X}_0- \bm{\overline{X}}_0|^2 = e^{-2\kappa t}\W_2^2(\mu_t,\mu_*). 
\end{equation*} 
{\ } \vskip-1.12cm \hfill\qedsymbol

\subsection{Proof of log-Sobolev inequality, Theorem \ref{th:LSI}} \label{se:LSI:proofs}

Fix $\mu_0 \in \P^{\otimes n}(\R^d)$ with $H(\mu_0\,|\,\rho_*) < \infty$. Let $\bm{X}$ be the solution of \eqref{def:mainSDE} with $\bm{X}_0 \sim \mu_0$, and let $\mu_t$ be the law of $\bm{X}_t$ for each $t > 0$.
Recall the computation \eqref{pf:entropycalculation}:
\begin{equation*}
\frac{d}{dt}\widetilde{H}(\mu_t\,|\,\rho_*) = \frac{d}{dt}H(\mu_t\,|\,\rho_*) = -\widetilde{I}(\mu_t\,|\,\rho_*) .
\end{equation*} 
Following the strategy of Bakry-\'Emery \cite{BakryEmery}, we will differentiate a second time and show that
\begin{equation}
\frac{d}{dt} \widetilde{I}(\mu_t\,|\,\rho_*) \le -2\kappa \widetilde{I}(\mu_t\,|\,\rho_*). \label{pf:BakryEmery}
\end{equation}
Using \eqref{pf:BakryEmery}, the proof concludes quickly:
By Gr\"onwall's inequality, $\widetilde{I}(\mu_t\,|\,\rho_*) \le e^{-2\kappa t}\widetilde{I}(\mu_0\,|\,\rho_*)$. 
Theorem \ref{th:convex-limit}(iii) and Theorem \ref{th:main-MF} together imply $\lim_{t\to\infty}\widetilde{H}(\mu_t\,|\,\rho_*) = 0$. Hence,
\begin{align*}
\widetilde{H}(\mu_0\,|\,\rho_*) = \int_0^\infty \widetilde{I}(\mu_t\,|\,\rho_*)\,dt \le \frac{1}{2\kappa}\widetilde{I}(\mu_0\,|\,\rho_*).
\end{align*}

To aid in the proof of \eqref{pf:BakryEmery} we introduce some shorthand.
Define a product measure $\nu_t:=\nu^1_t\otimes\cdots\otimes\nu^n_t$ by taking $\nu^i_t$ to have density proportional to $x^i \mapsto \exp \E[f(\bm{X}_t)\,|\,X^i_t=x^i]$. Note that $\bm{X}_t \sim \mu_t$ in the expectation here and below. This way,
\begin{align*}
\nabla_i\log\nu_t(\bm{X}_t)=\E[\nabla_i f(\bm{X}_t)\,|\,X^i_t] = \E[\nabla_i\log \rho_*(\bm{X}_t)\,|\,X^i_t].
\end{align*}
Abbreviate $h^i_t=\log\mu^i_t/\nu^i_t$ and $h_t=\log \mu_t/\nu_t$, and note crucially that because $\mu_t$ and $\nu_t$ are product measures, $h_t(\bm{x})=h^1_t(x^1)+\cdots+h^n_t(x^n)$ is additively separable. In particular,
\begin{align}
\nabla_i h_t(\bm{x}) = \nabla h_t^i(x^i) = \nabla\log\mu^i_t(x^i) - \E[\nabla_i\log \rho_*(\bm{X}_t)\,|\,X^i_t=x^i]. \label{pf:LSI-0}
\end{align}
Let us note also for later use that
\begin{align*}
\nabla^2_i h_t(\bm{x}) = \nabla^2\log\mu^i_t(x^i) - \E[\nabla_i^2\log \rho_*(\bm{X}_t)\,|\,X^i_t=x^i]. 
\end{align*}
With this notation, we may write  
\begin{equation}
\widetilde{I}(\mu_t\,|\,\rho_*) = \int_{\R^{dn}}|\nabla h_t|^2\,\mu_t.\label{pf:BakryEmery2}
\end{equation} 
The integral here and below is with respect to Lebesgue measure. 
We will differentiate the right-hand side of \eqref{pf:BakryEmery2}, proceeding again formally and omitting the fairly standard mollification arguments.
The Fokker-Planck equation \eqref{def:FokkerPlanck} (with $\bm{b}=\nabla f$) for $\mu$ can be written compactly as $\partial_t\mu = \nabla \cdot \big(\mu \nabla h)$
\begin{align*}
\partial_t\mu = \nabla \cdot \big(\mu \nabla h).
\end{align*}
Using this, we have
\begin{align}
\frac{d}{dt}\int_{\R^{dn}}|\nabla h_t|^2\mu_t &= 2\int_{\R^{dn}}\big[  \nabla h_t \cdot \partial_t \nabla h_t - \langle \nabla h_t, \nabla^2h_t \nabla h_t\rangle\big]\mu_t. \label{pf:LSI-1}
\end{align}
We simplify the first term of \eqref{pf:LSI-1} as follows. Note that
\begin{align}
\partial_t \nabla_i h_t = \partial_t \nabla \log \mu^i_t - \partial_t \E[\nabla_i f(\bm{X}_t)\,|\,X^i_t=\cdot]. \label{pf:LSI-2}
\end{align}
First, we have
\begin{align*}
\partial_t \nabla \log \mu_t &= \nabla \frac{\partial_t\mu_t}{\mu_t} = \nabla \bigg(\frac{1}{\mu_t} \nabla \cdot \big(\mu_t  \nabla h_t \big)\bigg) \\
	&= \nabla(\nabla \log \mu_t \cdot \nabla h_t) + \nabla \Delta h_t \\
	&= \nabla^2\log \mu_t \nabla h_t + \nabla^2 h_t \nabla\log\mu_t + \nabla \Delta h_t.
\end{align*}
By integration by parts,
\begin{align*}
\int_{\R^{dn}}(\nabla h_t \cdot\nabla \Delta h_t)\mu_t &= \sum_{i,j=1}^{dn}\int_{\R^{dn}}\partial_i h_t \partial_{ijj}h_t \mu_t \\
	&= -\sum_{i,j=1}^{dn}\int_{\R^{dn}}\partial_{ij} h_t \partial_{ij}h_t \mu_t  - \sum_{i,j=1}^{dn}\int_{\R^{dn}}\partial_i h_t \partial_{ij}h_t \partial_j \mu_t  \\
	&= -\int_{\R^{dn}}\|\nabla^2h_t\|_{\mathrm{Frob}}^2\mu_t - \int_{\R^{dn}}\langle \nabla h_t,\nabla^2h_t\nabla\log\mu_t\rangle\,\mu_t.
\end{align*}
Hence,
\begin{align}
\int_{\R^{dn}} \nabla h_t \cdot \partial_t\nabla\log \mu_t \,\mu_t &= \int_{\R^{dn}}\Big(\langle \nabla h_t,\nabla^2\log \mu_t \nabla h_t\rangle + \langle \nabla h_t,\nabla^2 h_t \nabla\log\mu_t\rangle \\
	&\qquad\qquad - \|\nabla^2h_t\|_{\mathrm{Frob}}^2 - \langle \nabla h_t,\nabla^2h_t\nabla\log\mu_t\rangle\Big)\mu_t \nonumber \\
		&= \int_{\R^{dn}}\Big(\langle \nabla h_t,\nabla^2\log \mu_t \nabla h_t\rangle - \|\nabla^2h_t\|_{\mathrm{Frob}}^2 \Big)\mu_t. \label{pf:LSI-3}
\end{align}
We next simplify the second term of \eqref{pf:LSI-2}, by computing
\begin{align*}
\partial_t \E[\nabla_i f(\bm{X}_t)\,|\,X^i_t=x^i] &=  \int_{\R^{d(n-1)}} \nabla_i f(x^i,\bm{x}^{-i})\,\partial_t \mu_t^{-i}(\bm{x}^{-i})\,d\bm{x}^{-i} \\
	&= \int_{\R^{d(n-1)}} \nabla_i f(x^i,\bm{x}^{-i})   \nabla_{-i} \cdot \big(\mu_t^{-i} \nabla_{-i}h_t\big)(\bm{x}^{-i}) \,d\bm{x}^{-i} \\
	&= - \sum_{j \neq i}\int_{\R^{d(n-1)}} \nabla_i\nabla_j^\top  f(x^i,\bm{x}^{-i}) \nabla_j^\top h_t(\bm{x}^{-i}) \,\mu_t^{-i}(\bm{x}^{-i})d\bm{x}^{-i},
\end{align*}
where $\nabla_{-i}$ denotes the gradient with respect to $(x^j)_{j \neq i}$.
Hence,
\begin{align}
\sum_{i=1}^n\E[\nabla_i h_t \cdot \partial_t \E[\nabla_i f(\bm{X}_t)\,|\,X^i_t]] &= - \sum_{i=1}^n\sum_{j \neq i}\int_{\R^{dn}} \langle \nabla_i h_t,\nabla_i\nabla_j^\top  f\nabla_j h_t\rangle \,\mu_t. \label{pf:LSI-4}
\end{align}
Plugging \eqref{pf:LSI-3} and \eqref{pf:LSI-4} into \eqref{pf:LSI-2}, we have
\begin{align*}
\int_{\R^{dn}}\nabla h_t \cdot \partial_t\nabla h_t\,\mu_t &= \int_{\R^{dn}}\Big(\langle \nabla h_t,\nabla^2\log \mu_t \nabla h_t\rangle - \|\nabla^2h_t\|_{\mathrm{Frob}}^2 + \sum_{i=1}^n\sum_{j \neq i}  \langle \nabla_i h_t,\nabla_i\nabla_j^\top f\nabla_j h_t\rangle\Big)\,\mu_t.
\end{align*}
Noting that $\nabla^2h_t=\nabla^2\log\mu_t - \nabla^2\log\nu_t$, we thus return to \eqref{pf:LSI-1} to find
\begin{align*}
\frac{d}{dt}\int_{\R^{dn}}|\nabla h_t|^2\mu_t &= 2\int_{\R^{dn}}\Big(  \langle \nabla h_t, \nabla^2\log\nu_t \nabla h_t\rangle - \|\nabla^2h_t\|_{\mathrm{Frob}}^2 + \sum_{i=1}^n\sum_{j \neq i}  \langle \nabla_i h_t,\nabla_i\nabla_j^\top f\nabla_j h_t\rangle\Big)\mu_t.
\end{align*}
By definition of $\nu_t$, the matrix $\nabla^2\log \nu_t(\bm{x})$ is block-diagonal with $i$th block given by $\E[\nabla_i^2f(\bm{X}_t)\,|\,X^i_t=x^i]$. Recalling \eqref{pf:LSI-0}, we find that
\begin{align*}
\int_{\R^{dn}} \langle \nabla h_t, \nabla^2\log \nu_t \nabla h_t\rangle  \,\mu_t &= \sum_{i=1}^n\E\big[\langle \nabla h^i_t(X^i_t), \E[\nabla_i^2f(\bm{X}_t)\,|\,X^i_t] \nabla h^i_t(X^i_t)\rangle\big] \\
	&= \sum_{i=1}^n\E\big[\langle \nabla h^i_t(X^i_t),  \nabla_{i}^2f(\bm{X}_t)  \nabla h^i_t(X^i_t)\rangle\big].
\end{align*}
Hence,
\begin{align*}
\frac{d}{dt}\int_{\R^{dn}}|\nabla h_t|^2\mu_t &= 2\int_{\R^{dn}}\Big(   - \|\nabla^2h_t\|_{\mathrm{Frob}}^2 + \sum_{i=1}^n\sum_{j=1}^n  \langle \nabla_i h_t\nabla_i\nabla_j f,\nabla_j h_t\rangle\Big)\mu_t \\
	&= 2\int_{\R^{dn}}\Big(   - \|\nabla^2h_t\|_{\mathrm{Frob}}^2 +  \langle \nabla h_t, \nabla^2 f\nabla h_t\rangle\Big)\mu_t.
\end{align*}
Discard the negative term and use the assumed concavity $\nabla^2 f \le -\kappa I$ to get
\begin{align*}
\frac{d}{dt}\int_{\R^{dn}}|\nabla h_t|^2\mu_t &\le - 2\kappa \int_{\R^{dn}}|\nabla h_t|^2\mu_t,
\end{align*}
which  proves \eqref{pf:BakryEmery} because of the identity \eqref{pf:BakryEmery2}. \hfill\qedsymbol

\section{Proof of Theorem \ref{th:entropicoptimality} on entropic optimality} \label{se:proof:entropicoptimality}

In this section, we work with $n$ copies of the canonical space $\Omega^i = C(\R_+;\R^d)$, for $i=1,\ldots,n$. On $\Omega^i$ we let $X^i$ denote the coordinate process, $X^i_t(\omega)=\omega_t$, and we let $\FF^i=(\F^i_t)_{t \ge 0}$ denote its filtration. Let $\Omega=\Omega^1 \times \cdots \times \Omega^n$ denote the product space, with each process $X^i$ and filtration $\FF^i$ lifting to $\Omega$ in the obvious manner. Let $\bm{X}=(X^1,\ldots,X^n)$.

We begin by proving \eqref{eq:entopt-identity}. The entropy identity of Lemma \ref{le:entropy-pathspace} shows that
\begin{equation*}
H(\mu[t]\,|\,\rho[t]) = H(\mu_0\,|\,\rho_0) + \frac14 \int_0^t \sum_{i=1}^n\E_\mu\big|\E_\mu[b^i(\bm{X}_s)\,|\,X^i_s]- b^i(\bm{X}_s)\big|^2\,ds.
\end{equation*}
Indeed, the condition \eqref{asmp:entropy-girsanov} for Lemma \ref{le:entropy-pathspace} holds here because $\nabla f$ has linear growth (being Lipschitz) and because $\mu_t$ and $\rho_t$ have bounded second moments for bounded time sets (see Proposition \ref{pr:wellposed}).
The quantity
\begin{align*}
\E_\mu\big|\E_\mu[b^i(\bm{X}_s)\,|\,X^i_s]- b^i(\bm{X}_s)\big|^2
\end{align*}
is finite for each $s$ by square-integrability of $\bm{X}_s$, and we will see later that it depends continuously on $s$. This implies
\begin{equation*}
\H'_0(\mu\,|\,\rho) = \lim_{t \downarrow 0}\frac{H(\mu[t]\,|\,\rho[t]) - H(\mu_0\,|\,\rho_0)}{t} = \frac14 \sum_{i=1}^n\E_\mu\big|\E_\mu[b^i(\bm{X}_0)\,|\,X^i_0]- b^i(\bm{X}_0)\big|^2,
\end{equation*}
which proves \eqref{eq:entopt-identity}.

To prove \eqref{ineq:entopt}, fix $\nu \in \P^{\otimes n}(C(\R_+;\R^d))$ with $\nu_0=\mu_0$. Let $\nu^i \in \P(C(\R_+;\R^d))$ denote the $i$th marginal for each $i$. Similarly, let $\rho^i \in \P(C(\R_+;\R^d))$ denote the $i$th marginal of $\rho$.

If $H(\nu[t]\,|\,\rho[t]) =\infty$ for all $t > 0$, then trivially $\H'_0(\nu\,|\,\rho) = \infty$, and there is nothing to prove. We may thus assume that $H(\nu[T]\,|\,\rho[T]) < \infty$ for some $T > 0$. It follows that $H(\nu[t]\,|\,\rho[t]) \le H(\nu[T]\,|\,\rho[T]) < \infty$ for all $t \in [0,T]$.
In addition, we have $H(\nu^i[t]\,|\,\rho^i[t]) < \infty$ for all $i$.

We next apply Girsanov's theorem, in the form of \cite{leonard2012girsanov} which is particularly well-suited to our needs, as it is designed under a finite-entropy assumption rather than standard but restrictive conditions like Novikov's. Indeed, the finite entropy $H(\nu^i[T]\,|\,\rho^i[T]) < \infty$ ensures by \cite[Theorem 2.1]{leonard2012girsanov} that there exists a $\FF^i$-progressively measurable process $\beta^i$ on $\Omega^i$, square-integrable under $dt \otimes d\nu^i$, such that 
\begin{align*}
X^i_t = \beta^i_tdt + \sqrt{2} \,d\widetilde{B}^i_t, \quad t \in (0,T),
\end{align*}
where $\widetilde{B}^i$ is a $\FF^i$-Brownian motion under $\nu^i$.
We may again lift the processes $\beta^i$ and $\widetilde{B}^i$ to the product space $\Omega$.
Using the entropy identity of \cite[Theorem 2.3]{leonard2012girsanov}, we have
\begin{align*}
H(\nu[t]\,|\,\rho[t]) = H(\nu_0\,|\,\rho_0) + \frac14 \int_0^t \sum_{i=1}^n\E_\nu\big|\beta^i_s - b^i(\bm{X}_s)\big|^2\,ds.
\end{align*}
Because $\beta^i_s$ is $\F^i_s$-measurable, we have
\begin{align}
\E_\nu\big|\beta^i_s - b^i(\bm{X}_s)\big|^2 &= \E_\nu\big|\beta^i_s - \E_\nu[b^i(\bm{X}_s)\,|\,\F^i_s]\big|^2 + \E_\nu\big|\E_\nu[b^i(\bm{X}_s)\,|\,\F^i_s] - b^i(\bm{X}_s)\big|^2 \nonumber \\
	&\ge \E_\nu\big|\E_\nu[b^i(\bm{X}_s)\,|\,\F^i_s] - b^i(\bm{X}_s)\big|^2. \label{pf:entopt-ineq1}
\end{align}
Combining the last two displays,
\begin{align*}
\H'_0(\nu\,|\,\rho) &\ge \liminf_{t \downarrow 0} \frac{1}{4t} \int_0^t \sum_{i=1}^n\E_\nu\big|\E_\nu[b^i(\bm{X}_s)\,|\,\F^i_s] - b^i(\bm{X}_s)\big|^2\,ds.
\end{align*}
If we show that the right-hand side of \eqref{pf:entopt-ineq1} is continuous in $s$, then, recalling that $\nu_0=\mu_0$, we deduce
\begin{align*}
\H'_0(\nu\,|\,\rho) &\ge \frac14 \sum_{i=1}^n\E_\mu\big|\E_\mu[b^i(\bm{X}_0)\,|\,X^i_0] - b^i(\bm{X}_0)\big|^2.
\end{align*}
Combined with \eqref{eq:entopt-identity}, this proves \eqref{ineq:entopt}.

It remains to show that the right-hand side of \eqref{pf:entopt-ineq1} is continuous in $s$, which was also used earlier in the proof in the case $\nu=\mu$. To prove this, note that the square-integrability of $\beta^i$ easily implies $\E_\nu|X^i_s-X^i_t|^2 \to 0$ as $t \to s$. Since $b^i$ is Lipschitz, it follows that $\E_\nu|b^i(\bm{X}_s)-b^i(\bm{X}_t)|^2 \to 0$ as $t \to s$. Note by independence that
\begin{align*}
\E_\nu[b^i(\bm{X}_s)\,|\,\F^i_s] &= \E\big[b^i(x^i,\bm{X}^{-i}_s)\big]\big|_{x^i=X^i_s} = \E_\nu[b^i(\bm{X}_s)\,|\,\F^i_t],
\end{align*}
which easily implies a similar $L^2$-continuity for $s\mapsto \E_\nu[b^i(\bm{X}_s)\,|\,\F^i_s]$. \hfill \qedsymbol

\begin{remark} \label{re:equalitycase}
The equality cases of \eqref{ineq:entopt} are easily identified from the above proof:  In order to have equality, it must be the case that the inequality \eqref{pf:entopt-ineq1} collapses to equality in the small-$t$ limit. That is, equality holds in \eqref{pf:entopt-ineq1} if and only if
\begin{align*}
\liminf_{t\downarrow 0} \frac{1}{t}\int_0^t  \E_\nu\big|\beta^i_s - \E_\nu[b^i(\bm{X}_s)\,|\,\F^i_s]\big|^2\,ds = 0 , \qquad \forall i=1,\ldots,n.
\end{align*}
In the case that $s \mapsto \beta^i_s \in L^2(\nu)$ is continuous at $s=0$, this implies that $\beta^i_0=\E_\mu[b^i(\bm{X}_0)\,|\,X^i_0]$ a.s.
\end{remark}

\section{Proofs for the proximity of the independent projection}

\subsection{Proof of Theorem \ref{th:proximity}}

A standard entropy identity computes the path-space relative entropy $H(\mu[T]\,|\,\rho[T])$ in terms of the squared $L^2(\mu[T])$-norm between their drifts, plus the time-zero entropy; see Lemma \ref{le:entropy-pathspace}. In our setting, this yields
\begin{align*}
H(\mu[T]\,|\,\rho[T]) &= H(\mu_0\,|\,\rho_0) + \frac14\int_0^T\sum_{i=1}^n\E\big| b^i(\bm{X}_t) - \E[b^i(\bm{X}_t)\,|\,X^i_t]\big|^2dt.
\end{align*}
This can be rewritten as a conditional variance, where the variance of a random vector is defined as the sum of the variances of the coordinates:
\begin{align*}
H(\mu[T]\,|\,\rho[T]) &= H(\mu_0\,|\,\rho_0) + \frac14\int_0^T\sum_{i=1}^n\E\,\Var(b^i(\bm{X}_t)\,|\,X^i_t)dt.
\end{align*}
It is shown in \cite[Theorem 4.2]{CattiauxGuillin} that $\mu^i_t$, as the law of an SDE with $L$-Lipschitz drift, satisfies a Poincar\'e inequality with constant $c_t$, for each $i$. The product measure $\mu^{-i}_t$ thus satisfies the Poincar\'e inequality with the same constant, and the claim \eqref{ineq:proximity1} follows. \hfill \qedsymbol

\subsection{Proof of Theorem \ref{th:proximity-uniform}}

We make use of a well known calculation of the time-derivative of the relative entropy between solutions of two different Fokker-Planck equations, similar to that of Section \ref{se:longtime:proofs}. Formally, if one is given solutions $\nu^i$ of Fokker-Planck equations $\partial_t\nu^i = -\mathrm{div}(\nu^ib^i) + \Delta \nu^i$ with different velocity fields $b^i$, for $i=1,2$, then
\begin{equation*}
\frac{d}{dt}H(\nu^1_t\,|\,\nu^2_t) = \int \bigg[ (b^1 - b^2) \cdot \nabla \log \frac{d\nu^1_t}{d\nu^2_t} - \Big|\nabla \log \frac{d\nu^1_t}{d\nu^2_t}\Big|^2\bigg]\,d\nu^1_t.
\end{equation*}
Using Young's inequality and integrating, this implies for any $t > s \ge 0$ that
\begin{equation}
H(\nu^1_t\,|\,\nu^2_t) + \frac{1}{2}\int_s^t I(\nu^1_u\,|\,\nu^2_u) \,du \le H(\nu^1_s\,|\,\nu^2_s) + \frac{1}{2}\int_s^t \int |b^1_u-b^2_u|^2\,d\nu^1_u  \,du, \label{pf:proximity2-entropy}
\end{equation}
where we recall that relative Fisher information $I$ was defined in \eqref{def:FisherInfo}.
The inequality \eqref{pf:proximity2-entropy} can be made rigorous and general more easily than the preceding identity, via mollification and lower semicontinuity arguments; see \cite[Lemma 2.4]{Bogachev2016} for a version under the sole assumption that $(b^1,b^2)$ are measurable and locally bounded, or the discussion around \cite[Lemma 3.1]{LackerLeFlem}.
In our context, \eqref{pf:proximity2-entropy} specializes to
\begin{equation*}
 H(\mu_t\,|\,\rho_t) + \frac{1}{2}\int_s^t I(\mu_u\,|\,\rho_u)\,du \le H(\mu_s\,|\,\rho_s) + \frac{1}{2}\sum_{i=1}^n\int_s^t \E\big[ \big|\E[\nabla_i f(\bm{X}_u)\,|\,X^i_u] - \nabla_i f(\bm{X}_u)\big|^2\big]\,du,
\end{equation*}
Recall that the measure flow $(\rho_u)_{u \ge 0}$ arises from the Markov process with infinitesimal generator $\varphi \mapsto \nabla f\cdot \nabla \varphi + \Delta \varphi$. The $\kappa$-concavity of $f$ ensures that $\rho_t$ thus satisfies log-Sobolev inequality $H(\cdot\,|\,\rho_t) \le (1/ 2\eta)I(\cdot\,|\,\rho_t)$, where $\eta=\min(\kappa,\eta_0)$.
Indeed, this variant of the famous Bakry-\'Emery result may be found in \cite[Corollary 3.7]{malrieu2001logarithmic}, after noting that
\[
1/2\eta = \max(1/ 2\kappa,1/2\eta_0 ) \ge \frac{1}{2\kappa}(1-e^{-2\kappa t}) + \frac{1}{2\eta_0}e^{-2\kappa t},
\]  
where the right-hand side is the sharper (but time-dependent) constant given by \cite[Corollary 3.7]{malrieu2001logarithmic}.
Hence,  
\begin{align*}
 H(\mu_t\,|\,\rho_t) + \eta \int_s^tH(\mu_u\,|\,\rho_u)\,du &\le H(\mu_s\,|\,\rho_s) + \frac{1}{2}\sum_{i=1}^n\int_s^t \E\big[ \big|\nabla_i f(\bm{X}_u) - \E[\nabla_i f(\bm{X}_u)\,|\,X^i_u]\big|^2\big]\,du.
\end{align*}
The drift $b^i(u,x) := \E[\nabla_i f(\bm{X}_u)\,|\,X^i_u=x]$ of the $i$th component of the independent projection satisfies the dissipativity property
\begin{align*}
\big(b^i(u,y)-b^i(u,z)\big) \cdot (y-z) &= \int_{(\R^d)^{n-1}} \big(\nabla_i f(y,\bm{x}^{-i}) - \nabla_i f(z,\bm{x}^{-i})\big) \cdot (y-z) \,\mu_u^{-i}(d\bm{x}^{-i}) \le -\kappa |y-z|^2,
\end{align*}
for $y, z \in \R^d$.
It follows from \cite[Theorem 4.2]{CattiauxGuillin} that $\mu^i_t$  satisfies a Poincar\'e inequality with constant 
\[
e^{-2\kappa t}c_0 + \frac{1}{ \kappa}(1-e^{-2\kappa t}) \le \max(c_0,1/\kappa) = c.
\]
This is true for each $i$, and in particular $\mu_t^{-i}$ satisies a Poincar\'e inequality with the same constant.
Hence,
\begin{align*}
 H(\mu_t\,|\,\rho_t) + \eta \int_s^tH(\mu_u\,|\,\rho_u)\,du &\le H(\mu_s\,|\,\rho_s) + \frac{c}{2}\sum_{i=1}^n\sum_{j \neq i}\int_s^t \E\big[ \big|\nabla_{ij} f(\bm{X}_u)\big|^2\big]\,du.
\end{align*}
Apply Gronwall's inequality (precisely, in the form of \cite[Lemma B.1]{LackerLeFlem}) to deduce
\begin{align*}
H(\mu_t\,|\,\rho_t) &\le e^{- \eta  t }H(\mu_0\,|\,\rho_0) + \frac{c}{2}\sum_{i=1}^n\sum_{j \neq i}\int_0^te^{-  \eta (t-s) }\E\big[ \big|\nabla_{ij} f(\bm{X}_s)\big|^2\big]\,ds. 
\end{align*}
{ \ } \vskip-1.2cm \hfill \qedsymbol

\subsection{Proof of Corollary \ref{co:propchaos-uniform}}

The evenness of $V$ and symmetry of $A$ imply that $\nabla_{ij}f(\bm{x}) =  A_{ij}\nabla^2 V(x^i-x^j)$, and it follows from the nonnegativity of the entries of $A$  that the function $f$ is $\kappa$-concave.
As noted just above in the proof of Theorem \ref{th:proximity-uniform}, $\rho_t$ satisfies the LSI with constant $\eta$, for all $t \ge 0$. It was shown by Otto and Villani \cite{otto2000generalization} this this implies the transport inequality
\begin{equation*}
\W_2^2(\nu,\rho_t) \le \frac{2}{\eta}H(\nu\,|\,\rho_t), \quad \forall \nu \in \P((\R^d)^n).
\end{equation*}
Combined with the subadditivity inequality \eqref{def:subadditivity}, this yields the first claimed inequality. To prove the second, we simply apply Theorem \ref{th:proximity-uniform}. \hfill \qedsymbol

\bibliographystyle{amsplain}

\bibliography{biblio}
\end{document}